\documentclass{article}
\usepackage{authblk}
\usepackage[all]{xy}
\usepackage{mathrsfs}
\usepackage{amsthm,amsmath}
\usepackage{geometry}
\usepackage[all]{xy}
\DeclareMathOperator{\rk}{rk}
\DeclareMathOperator{\ord}{ord}

\newdir{ >}{{}*!/-5pt/@{>}}
\SelectTips {cm}{}
\usepackage{mathrsfs}
\usepackage{amsthm,amsmath,amssymb,color,enumitem}
\usepackage{setspace}
\theoremstyle{plain}
\newtheorem{thm}{Theorem}[section]
\newtheorem{assum}[thm]{Assumption}
\newtheorem{nota}[thm]{Notation}
\theoremstyle{definition}
\newtheorem{defn}{Definition}[section]
\newtheorem{prop}{Proposition}[section]

\newtheorem{lemma}{Lemma}[section]
\newtheorem{eg}{Example}[section]
\newtheorem{rmk}{Remark}[section]
\theoremstyle{examples}
\newtheorem{egs}{Examples}[section]

\newcommand{\ndot}{\raisebox{.4ex}{.}}
\def\norm#1{\mathopen{\|}#1\mathclose{\|}}
\title{Categorification of Harder-Narasimhan Theory via Slope Functions in Totally Ordered Sets  }
\author{ Yao Li}
\affil{School of Mathematical Sciences, Peking University, China}
\begin{document}
\maketitle
\begin{abstract}
We introduce a categorical construction of Harder-Narasimhan filtration via a slope method which does not need a degree function. With a theorem of existence and uniqueness of Harder-Narasimhan filtration in our categorical setting, we give a categorical interpretation of Stuhler-Grayson filtration in the case of non-necessarily Hermitian normed lattices.
\end{abstract}

\maketitle
\section{Introduction}
The notion of Narasimhan filtration was first introduced by Harder and Narasimhan \cite{MR364254} in the setting of vector bundles on a non-singular projective curve. Let $k$ be an algebraic closed field and $C$ be a non-singular projective curve over $k$, they proved that any non-zero vector bundle $E$ on $C$ admits a canonical filtration by vector subbundles
\[0=E_0\subsetneq E_1\subsetneq\ldots\subsetneq E_n=E,\]
where the subquotients $E_i/E_{i-1}$ are semi-stable vector bundles with strictly decreasing slopes. Later Harder-Narasimhan filtration has been generalized to the setting of pure sheaves over higher dimensional polarized projective varieties, see \cite[\S1.3]{MR2665168} for more details. Curiously, analogous constructions have also been  discovered in many other branches of mathematics.
 For example, in geometry of numbers, Stuhler \cite{MR424707} proposed a construction of Harder-Narasimhan filtration for Euclidean lattices and compared it to the reduction theory of positive quadratic forms. Later Stuhler's result has been generalized in various settings of arithmetic geometry by Grayson \cite{MR780079}, Bost \cite{MR1423622}, Gaudron \cite{MR2431505} etc. Moreover, Faltings and
  W\"{u}stholz \cite{MR1253191} developed a Harder-Narasimhan theory on filtered vector spaces and applied it to diophantine approximation. Recently, their results on Harder-Narasimhan theory were expanded by Grieve \cite{G2022}. In the framework of arithmetic geometry over a function field, Lafforgue showed that a family of generalised chtoucas possesses a canonical filtration of Harder-Narasimhan type (see \cite{MR1600006} th\'eor\`eme 5). In arithmetic geometry of local fields, similar constructions exist for F-crystals (Dieudonn{\'e}-Manin filtration, see \cite{MR0157972} Chapter II, see also \cite{MR563463}), 
  filtered $\varphi$-modules (see \cite{MR1293972} and \cite[\S5.5.2.4]{MR3917141}), and also $\varphi$-modules over Robba ring \cite{MR2493220}. Moreover, Fargues \cite{MR2673421} developed a Harder-Narasimhan theory for finite flat group schemes over a valuation ring of unequal characteristic. In the theory of linear codes, a construction of Harder-Narasimhan filtration was proposed by Randriambololona \cite{MR3912956}.

These results in various context of quite different nature have motivated several categorical constructions of Harder-Narasimhan filtration. We can mention for example Lafforgue's construction (see \cite[\S II.2]{MR1600006}) and Bridgeland's construction \cite{MR2373143} on abelian categories, Chen's approach \cite{MR2559690} in the framework of exact categories equipped with a geometric structure, Andr\'e's approach \cite{MR2571693} for proto-abelian categories, and that of Cornut \cite{MR3928223} for modular lattices. All these categorical constructions have a common nature, 
which we resume as follows. One begins with a category $\mathcal C$ with a zero object on which two real valued functions $\deg(\ndot)$ and $\operatorname{rk}(\ndot)$ are given, where the function $\operatorname{rk}(\ndot)$ is assumed to take positive values on non-zero objects. These functions induce a slope function
\[\mu(\ndot)=\frac{\deg(\ndot)}{\operatorname{rk}(\ndot)}\]
on non-zero objects, which allows to define the semi-stability of objects in $\mathcal C$ : a non-zero object $Y$ of $\mathcal C$ is said to be semi-stable if, for any strict monomorphism $X\rightarrowtail Y$ with $X$ non-zero, one has $\mu(X)\leqslant\mu(Y)$. The functions $\deg(\ndot)$ and $\operatorname{rk}(\ndot)$ are supposed to be additive with respect to short exact sequences and verify a super-parallelogram condition, namely, for any object $Y$ in $\mathcal C$ and all strict monomorphisms $X_1\rightarrowtail Y$ and $X_2\rightarrowtail Y$, the following relations hold:
\begin{gather*}\deg(X_1+X_2)+\deg(X_1\cap X_2)\geqslant\deg(X_1)+\deg(X_2),\\
\operatorname{rk}(X_1+X_2)+\operatorname{rk}(X_1\cap X_2)=\operatorname{rk}(X_1)+\operatorname{rk}(X_2),\end{gather*}
where $X_1\cap X_2$ denotes the fiber product $X_1\times_YX_2$ and $X_1+X_2$ denotes the cofiber coproduct of $X_1$ and $X_2$ over $X_1\cap X_2$. Then, under these conditions, it can be shown that, for any non-zero object $Y$ of $\mathcal C$, there exists a  unique sequence of strict monomorphisms
\[\xymatrix{0=Y_0\ar@{ >->}[r]^-{f_1}&Y_1\ar@{ >->}[r]^-{f_2}&\;\cdots\;\ar@{ >->}[r]^-{f_n}&Y_n=Y},\]
such that all subquotients $Y_i/Y_{i-1}:=\operatorname{Coker}(f_i)$ are non-zero and semi-stable objects, and verify the following inequalities with respect to slopes:
\[\mu(Y_1/Y_0)>\ldots>\mu(Y_n/Y_{n-1}).\]

Curiously, in a recent work of Chen and Moriwaki \cite{Chen_Moriwaki}, an extension of semi-stability condition to the case of normed lattices (or more generally, adelic vector bundles) has been proposed so that there does exist, for each non-zero normed lattice, a Harder-Narasimhan type filtration with semi-stable subquotients and strictly decreasing slopes. 
Instead of the naive generalization  by comparing the slope of the normed lattice and those of its sublattices, they defined a non-zero normed lattice $\overline E$ to be semi-stable if and only if, for any non-zero sublattice $\overline F$ of $\overline E$, the inequality $\mu_{\min}(\overline F)\leqslant\mu_{\min}(\overline E)$ holds. Note that this condition is equivalent to the classical semi-stability condition when the norm $\|\ndot\|$ is Euclidean. 
Then they established a Harder-Narasimhan type theorem asserting that, for any normed lattice $\overline E$, there exists a unique sequence
\[0=E_0\subsetneq E_1\subsetneq\ldots\subsetneq E_n=E\]
of sublattices of $E$, such that all subquotient normed lattices $\overline{E_i/E_{i-1}}$ are semi-stable, and verify the following inequalities:
\[\mu_{\min}(\overline{E_1/E_0})>\ldots>\mu_{\min}(\overline{E_{n}/E_{n-1}}).\]
We refer the readers to \cite{Chen_Moriwaki} Theorem 4.3.58 for more details. This set-up of Harder-Narasimhan filtration is not included in any of the categorical constructions cited above. The main reason is that the degree function in the general normed lattices case is not additive with respect to short exact sequences, and hence does not verify the common requirement of these categorical models. This intriguing observation motivates the current work.

Note that the existence and uniqueness of Harder-Narasimhan filtration is  rather a statement concerning slopes of subquotients than that of the degree function. We refer the readers to \cite[Proposition 5.14]{MR2431505} for an explicit link between successive slopes and a minimax of minimal slopes of subquotients. In this paper, we try to give a conceptual explanation of this phenomenon in proposing a categorical formulation of Harder-Narasimhan theory based on slope functions. This choice of framework makes it possible to consider a general slope function taking values in an ordered set.  
Moreover, we do not require any compatibility of our slope function with respect to short exact sequence. Even the rank function is not needed (replaced by chain conditions of strict sub-object). Actually, we have the following result.
\begin{thm}{\label{Thm: in introduction}}
Let $E$ be a non-zero object of a small proto-abelian category $\mathcal{C}$ satisfying Noetherian and Artinian conditions on the chains of strict sub-objects of $E$, $\mu:S^0_E\rightarrow \Lambda$ be a slope function, where $\Lambda$ is a complete totally ordered set. There exists a unique filtration of $E$,
 \[0=E_0\rightarrow E_1\rightarrow \cdots \rightarrow E_n=E,\] such that each $E_{i+1}/E_i$ is semi-stable and  $\hat{\mu}_1> \hat{\mu}_2>\cdots >\hat{\mu}_n$, where  $\hat{\mu}_i=\mu_{\min}(E_i/E_{i-1})$.
\end{thm}

If we do not require the uniqueness, the value of the slope function could be taken in a partially ordered set. We have the following result.

\begin{thm}\label{Thm: introdution 2}
Let $E$ be a non-zero object of a small proto-abelian $\mathcal{C}$, which satisfies Noetherian and Artinian conditions on the chains of strict sub-objects of $E$, and $\mu:S^0_E\rightarrow \Lambda$ be a slope function, where $\Lambda$ is a partially ordered set which admits a greatest element and any subset of $\Lambda$ admits a infimum.
There exists a filtration
\[\xymatrix{0=E_0\ar@{ >->}[r]& E_1\ar@{ >->}[r]&\cdots\cdots\ar@{ >->}[r]& E_n=E,}\]
which satisfies the following properties:
\begin{enumerate}
\item for any $i\in\{1,\ldots,n\}$, $E_{i}/E_{i-1}$ is semi-stable,
\item for all $i\in\{1,\ldots,n-1\}$, $\mu_{\min}(E_i/E_{i-1})\not\leqslant\mu_{\min}(E_{i+1}/E_{i})$.
\end{enumerate}
\end{thm}

Theorem \ref{Thm: in introduction} and \ref{Thm: introdution 2} are proposed in detail in Theorem \ref{existence} and \ref{Uniqueness}. Firstly, we explain in an example what the classical construction is and why we need to consider the minimal slope $\mu_{\min}$. Later, we will explain what we do to prove our main theorem.

Let us illustrate the classical  construction by the example of Euclidean lattices. We call \emph{Euclidean lattice} any pair $\overline E=(E,\norm{\ndot})$, where $E$ is a finite generated free abelian group and $\norm{\ndot}$ is a Euclidean norm on the real vector space $E_{\mathbb R}:=E\otimes_{\mathbb Z}\mathbb R$. The Euclidean lattices form a proto-abelian category, where a morphism from a Euclidean lattice $\overline E$ to another one $\overline F$ is by definition a group homomorphism $f:E\rightarrow F$ such that $f_{\mathbb R}:E_{\mathbb R}\rightarrow F_{\mathbb R}$ has an operator norm $\leqslant 1$. A morphism $f:\overline E\rightarrow\overline F$ is a strict monomorphism if and only if $f$ is an injective group homomorphism satisfying $F/E$ is free and the norm of $\overline E$ is the restriction of that of $\overline F$. It is a strict epimorphism if and only if $f$ is a surjective  group homomorphism and the norm of $\overline F$ is the quotient norm of that of $\overline E$.

Note that a rank function $\operatorname{rk}(\ndot)$ is naturally defined on the category of Euclidean lattices, which sends $\overline E$ to the rank of $E$ over $\mathbb Z$. Recall that the Arakelov degree of a Euclidean lattice is defined to be the opposite logarithm of the covolume of the lattice, namely
\begin{equation}\label{Equ: degree}\deg(E,\norm{\ndot}):=-\ln\norm{e_1\wedge\cdots\wedge e_r}_{\det},\end{equation}
where $(e_i)_{i=1}^r$ is a basis of $E$ over $\mathbb Z$ (the value of $\deg(E,\norm{\ndot})$ does not depend on the choice of the lattice basis), and $\norm{\ndot}_{\det}$ is the determinant norm on the exterior space $\Lambda^r(E_{\mathbb R})$, defined as
\[\|\eta\|_{\det}:=\inf_{\begin{subarray}{c}(s_i)_{i=1}^r\in E_{\mathbb R}^r\\
\eta=s_1\wedge\cdots\wedge s_r
\end{subarray}}\norm{s_1}\cdots\norm{s_r}.\]
It can be shown that the degree function $\deg(\ndot)$ is additive with respect to short exact sequences. In other words, if $\overline E$, $\overline F$ and $\overline G$ are three Euclidean lattices and if
\[\xymatrix{0\ar[r]&E\ar[r]&F\ar[r]&G\ar[r]&0}\]
is a short exact sequence of abelian groups such that the norm of $\overline E$ is the restriction of that of $\overline F$, and the norm of $\overline G$ is the quotient norm of that of $\overline F$, then the following equality holds:
\begin{equation}\label{Equ: degree exact sequence}\deg(\overline F)=\deg(\overline E)+\deg(\overline G).\end{equation}
Moreover, if $\overline E$ is a Euclidean lattice and if $E_1$ and $E_2$ are subgroups of $E$, then the inequality
\begin{equation}\label{Equ: concavity}\deg(\overline{E_1+E_2})+\deg(\overline{E_1\cap E_2})\geqslant\deg(\overline E_1)+\deg(\overline E_2)\end{equation}
holds, where on subgroups of $E$ we consider restricted norms (we refer the readers to \cite{MR424707} Proposition 2 for a proof). Joint with the classical relation
\[\operatorname{rk}(E_1+E_2)+\operatorname{rk}(E_1\cap E_2)=\operatorname{rk}(E_1)+\operatorname{rk}(E_2),\]
we may interpret this inequality in a geometric way as follows. Let
\begin{gather*}(r_0,d_0):=(\operatorname{rk}(E_1\cap E_2),\deg(\overline{E_1\cap E_2})),\\ (r_1,d_1):=(\operatorname{rk}(E_1),\deg(\overline{E_1})),\quad
(r_2,d_2):=(\operatorname{rk}(E_2),\deg(\overline{E_2})).
\end{gather*}
Let $(r_3,d_3)$ be the point such that $r_3=\operatorname{rk}(E_1+E_2)$ and that $(r_i,d_i)$ with $i\in\{0,1,2,3\}$, form the vertices of a parallelogram, then the inequality
\[\deg(\overline{E_1+E_2})\geqslant d_3\]
holds. Then it can be shown  that any Euclidean lattice $\overline E$ admits a unique non-zero sublattice $\overline E_{\operatorname{des}}$, such that, firstly,
\[\mu(\overline{E}_{\operatorname{des}})=\max_{0\neq \overline F\subseteq \overline E}\mu(\overline F),\]
where $\overline F$ runs over the set of non-zero sublattices of $\overline E$, secondly, for any non-zero sublattice $\overline{F}$ of $\overline{E}$ satisfying $\mu(\overline F)=\mu(\overline E_{\operatorname{des}})$, one has $F\subseteq E_{\operatorname{des}}$. Note that $\overline E$ is semi-stable if and only if
\[\overline{E}_{\operatorname{des}}=\overline E.\]
Therefore, $\overline{E}_{\operatorname{des}}$ is called the destabilizing sublattice of $\overline E$. The Harder-Narasimhan filtration
\[0=E_0\subsetneq E_1\subsetneq\ldots\subsetneq E_n=E\]
can be constructed in a recursive way such that
\[\overline{E_i/E_{i-1}}=(\overline{E/E_{i-1}})_{\operatorname{des}}.\]
Note that each subquotient $\overline{E_i/E_{i-1}}$ is a non-zero semi-stable Euclidean lattice. Moreover, the following inequalities hold:
\[\mu(\overline{E_1/E_0})>\ldots>\mu(\overline{E_n/E_{n-1}}).\]
The first term $\mu(\overline{E_1/E_0})$ and the last term $\mu(\overline{E_n/E_{n-1}})$ in the above inequalities are called maximal slope and minimal slope of $\overline E$, denoted by $\mu_{\max}(\overline E)$ and $\mu_{\min}(\overline E)$ respectively. These invariants have more specific interpretation: the maximal slope is the maximal value of slopes of sublattices, and the minimal slope is the minimal value of slopes of quotient lattices. By the additivity of the degree function with respect to short exact sequences, the following equality holds for any non-zero Euclidean lattices,
\[\mu_{\min}(\overline E{}^\vee)=-\mu_{\max}(\overline E).\]
Moreover, $\overline E$ is semi-stable if and only if $\mu(\overline E)=\mu_{\max}(\overline E)$, or equivalently $\mu(\overline E)=\mu_{\min}(\overline E)$.

In Arakelov geometry, the above construction was generalized to adelic vector bundles by Gaudron \cite{MR2431505}, allowing for example lattices in a general normed vector space. By using Harder-Narasimhan polygons, it is formally possible to extend the definition of successive slopes in this framework. For simplicity, we describe the construction in the setting of \emph{normed lattices}, namely a couple $\overline E=(E,\norm{\ndot})$ consisting of a finite generated free abelian group, equipped with a norm on $E_{\mathbb R}$. Note that the degree function is defined in the same way as in \eqref{Equ: degree}. 
However, in general the relations \eqref{Equ: degree exact sequence} is no longer true. We give an example. Consider subgroup $E$ of additive group $\mathbb{C}$ generated by $1$ and $i$. We canonically identify $\mathbb{C}$ with $\mathbb{R}^2$ and equip it with $\ell^1$-norm.
Consider the norm lattice $\overline E=(E,\|\|_{\ell^1})$. Let
$F$ be the subgroup of $E$ generated by $1+i$ and $\overline F$ be a sublattice of $\overline E$. Then
\[\deg(\overline F)+\deg(\overline {E/F})=-\ln2-\ln1=-\ln2.\]
However, $\deg(\overline E)=0$.

Geometrically, we could consider the convex hull of the points in the plane $\mathbb R^2$ of coordinates $(\operatorname{rk}(F),\deg(\overline F))$, where $\overline F$ runs over the set of all sublattices of $\overline E$. Then the upper boundary of this convex hull is the graph of a concave and piecewise affine function $P_{\overline E}$, called \emph{Harder-Narasimhan polygon} of $\overline E$. 
Moreover, the abscissas where the function changes slopes are integers between $0$ and $\operatorname{rk}(E)$. In the case where the norm $\norm{\ndot}$ is Euclidean, the successive slopes of the Harder-Narasimhan polygon $P_{\overline E}$ are equal to those of successive subquotients of the Harder-Narasimhan filtration of $\overline E$. However, if $\norm{\ndot}$ is not Euclidean, in general it is not possible to construct a  sequence of sublattices of $\overline E$ with semi-stable subquotients corresponding to the same successive slopes of $P_{\overline E}$, 
which makes Harder-Narasimhan filtration special in this case. Let us illustrate this phenomenon by an example as follows. Let $k$ be a real number such that $k>\sqrt{2}$. We consider the vector space $\mathbb{R}^2$ equipped with the norm $\|\ndot\|_k$ defined as follows:
 For any $(a,b)\in \mathbb{R}^2$,
 \[\|(a,b)\|_k=(\max{\{a,kb,0\}}^2+\min{\{a,kb,0\}}^2)^\frac{1}{2}.\]
Consider the lattice in $\mathbb R^2$ generated by $e_1=(1,0)$ and $e_2=(0,1)$. One has
\[\|e_1\wedge e_2\|_{k,\det}=\frac{\|ke_1+e_2\|_k\cdot\|e_1-k^{-1}e_2\|_k}{2}=\frac{k}{\sqrt{2}},\]
which shows that the Arakelov degree of $(\mathbb Z^2,\|\ndot\|_k)$ is $\frac 12\log(2)-\log(k)$. Moreover, $(\pm 1,0)$ are non-zero lattice points having the smallest norm $1$ with respect to $\|\ndot\|_k$. Therefore, $\overline{\mathbb Ze_1}$ is the only sublattice of $(\mathbb Z^2,\|\ndot\|_k)$ whose Arakelov degree (which is zero) coincides with the maximal slope of $(\mathbb Z^2,\|\ndot\|_k)$. In particular, the last slope of the Harder-Narasimhan polygon of $(\mathbb Z^2,\|\ndot\|_k)$ is $\frac 12\log(2)-\log(k)$. However, the quotient norm $\|\ndot\|_{k,\operatorname{quot}}$ on $\mathbb Z^2/\mathbb Ze_1$ satisfies
\[\|[e_2]\|_{k,\operatorname{quot}}=\inf_{\lambda\in\mathbb R}\|e_2+\lambda e_1\|=\inf_{\lambda\in\mathbb R}\Big(\max\{\lambda,k,0\}^2+\min\{\lambda,k,0\}^2\Big)^{\frac 12}=k.\]
Therefore, the Arakelov degree of $(\mathbb Z^2/\mathbb Ze_1,\|\ndot\|_{k,\operatorname{quot}})$ is $-\log(k)$, which does not coincide with the last slope of the Harder-Narasimhan polygon of $(\mathbb Z^2,\|\ndot\|_k)$.

In this article, we prove the existence and uniqueness of Harder-Narasimhan filtration by a method of slope functions. The framework of this article is the theory of proto-abelian category, 
proposed by Andr\'e \cite{MR2571693}. In section 2 and 3, we recall the concept and some results of proto-abelian category following the work of Andr\'e \cite{MR2571693}. Then we give a universal characterization (see Proposition \ref{properties}) of the sum of two strict sub-objects of an object $E$ in a proto-abelian category which means the sum of two strict sub-objects could be considered as the smallest object larger than both of the two objects. In section 4, we define the slope functions of an object $E$ (see Definition \ref{Def: slope function}) and show that the slope function of $E$ could induce a slope function of its subquotient in a canonical way (see Proposition \ref{prop:compactilblity}). 
In section 5, we prove Proposition \ref{main} and Lemma \ref{main2} which give a different method to find the destabilizing sub-object. With the help of Noetherian and Artinian conditions, we can prove the Theorem \ref{existence} and Theorem \ref{Uniqueness} which are the main theorems. It should be noted that, in Proposition \ref{main} and Lemma \ref{main2}, we ask the slope function satisfying the strong slope inequality, which means the slope function considered here is actually $\mu_{\max}$ in classical reference by Proposition \ref{prop: mu max}. However, the minimal slope is not affected in the classical case by \cite[p. 256]{Chen_Moriwaki}.

We also give a method to prove the existence and uniqueness theorem of the Harder-Narasimhan filtration given a bounded lattice $(\Gamma, \leqslant)$. This method is more general than the method in proto-abelian category and we could get Harder-Narasimhan filtration escaping from the frame of proto-abelian category.  However, the main theorem is stated in the framework of proto-abelian category because it is general enough for application.

\section{Strict monomorphisms and epimorphisms}
In this section, we recall some basic properties of strict monomorphisms and strict epimorphisms, using Andr\'e's terminology \cite{MR2571693}.
\begin{defn}
Let $\mathcal{C}$ be a category having a zero object {$\boldsymbol{0}$}. We call a morphism in $\mathcal{C}$ a \emph{strict monomorphism} if it is a kernel of some morphism. We often use a tailed arrow $\rightarrowtail$ to denote a monomorphism. By the universal property of kernel, we can show that a strict monomorphism is a monomorphism. Similarly, we call a morphism in $\mathcal{C}$ a \emph{strict epimorphism} if it is a cokernel of a morphism. We often use a two headed arrow $\twoheadrightarrow$ to denote an epimorphism. Any strict epimorphism is necessarily an epimorphism.
\end{defn}

\begin{lemma}\label{Lem: kernel cokernel}Let $\mathcal C$ be a category having a zero object $\boldsymbol{0}$.
\begin{enumerate}[label=\rm(\arabic*)]
\item\label{Item: kernel of cokernel} Let $f:M\rightarrow N$ be a strict monomorphism of $\mathcal C$. If $f$ admits a cokernel $g:N\rightarrow Q$, then it is a kernel of $g$.
\item\label{Item: cokernel of kernel} Let $f':N'\rightarrow M'$ be a strict epimorphism of $\mathcal C$. If $f'$ admits a kernel $g':Q'\rightarrow N'$, then it is a cokernel of $g'$.
\end{enumerate}
\end{lemma}
\begin{proof}
\ref{Item: kernel of cokernel} Let $h:N\rightarrow P$ be a morphism of $\mathcal C$ such that $f$ is a kernel of $h$. Since $h\circ f=0$, there exists a unique morphism $u:Q\rightarrow P$ such that $u\circ g=h$ by the universal property of cokernel.
\[\xymatrix{&R\ar@{.>}[ld]_-{\exists!v}\ar[d]^-j\\
M\ar[r]_-f&N\ar[r]^-g\ar[d]_-h&Q\ar[ld]^-{\exists !u}\\
&P}\]
Let $j:R\rightarrow N$ be a morphism such that $g\circ j=0$. Then one has $h\circ j=u\circ g\circ j=0$. By the universal property of kernel, there exists a unique morphism $v:R\rightarrow M$ such that $j=f\circ v$. Therefore, $f$ is a kernel of $g$.

\ref{Item: cokernel of kernel} can be deduced from  \ref{Item: kernel of cokernel} by passing to the opposite category.
\end{proof}

\begin{defn}
A \emph{short exact sequence} of $\mathcal{C}$ is by definition a diagram of morphisms of $\mathcal{C}$ of the form
\[\xymatrix{\relax 0\ar[r]&L\ar[r]^-{f}&M\ar[r]^-g&N\ar[r]&0},\]
 such that $f$ is a kernel of $g$ and $g$ is a cokernel of $f$. By definition, if the above diagram is a short exact sequence, then $f$ is a strict monomorphism and $g$ is a strict epimorphism.
\end{defn}

\begin{defn}
Let $f:M\rightarrow N$ and $g:Q\rightarrow N$ be morphisms of $\mathcal C$. We call  \emph{pull-back of $g$ by $f$} any morphism $g':P\rightarrow M$ which fits into a cartesian diagram as follows:
\[\
\xymatrix{\relax
          P\ar[r]^-{f'} \ar[d]_-{g'}     &         Q \ar[d]^-{g}    \\
  M \ar[r]_-{f} & N         }\]
Similarly, if $f:M\rightarrow N$ and $h:M\rightarrow A$ are morphisms of $\mathcal C$, we call \emph{push-forward} of $h$ by $f$ any morphism $h'': N\rightarrow B$ which fits into a cocartesian diagram.
\[\xymatrix{\relax M\ar[r]^-f\ar[d]_-{h}&N\ar[d]^-{h''}\\
A\ar[r]_-{f''}&B}\]
\end{defn}

\begin{assum}
In the rest of the section, we fix a category $\mathcal C$ having a zero object $\boldsymbol{0}$. We also assume that any morphism of $\mathcal{C}$ admits a kernel and a cokernel.
\end{assum}

\begin{prop}\label{pull-back}
Let $f:M\rightarrow N$ be a strict monomorphism of $\mathcal C$. Then, for any morphism $g:Q\rightarrow N$, the pull-back of $f$ by $g$ always exists and is a strict monomorphism.
\end{prop}
\begin{proof}
 See \cite[p. 9]{MR2571693}.
\end{proof}

By passing to the opposite category, we deduce from Proposition \ref{pull-back}
the following result.
\begin{prop}\label{Pro: push-forward ep}
Let $f':P\rightarrow M$ be a strict epimorphism of $\mathcal C$. Then, for any morphism $g':P\rightarrow Q$, the push-forward of $f'$ by $g'$ always exists and is a strict epimorphism.
\end{prop}

\begin{nota}
Let $M\stackrel{f}\longrightarrow N\stackrel{g}\longleftarrow Q$ be a pair of  morphisms in $\mathcal C$, where either $f$ or $g$ is a strict monomorphism . We often denote by $M\times_N Q$ a fiber product of $f$ and $g$, and by
\[\operatorname{pr}_1:M\times_N Q\rightarrow M\quad\text{ and }\quad\operatorname{pr}_2:M\times_NQ\rightarrow Q\]
the universal morphisms.
\end{nota}

\begin{prop} \label{strict}
Let $f:L\rightarrow M$ and $g:M\rightarrow N$ be two morphisms of $\mathcal C$. Assume that  $g\circ f$ is a strict monomorphism and $g$ is  a monomorphism, then $f$ is a strict monomorphism.
\end{prop}
\begin{proof}
See \cite[p. 10]{MR2571693}.
\end{proof}

\begin{prop}\label{push-out}
Consider the following cartesian square in $\mathcal C$.\begin{equation}\label{Equ: square cartesian}\begin{gathered}\xymatrix{
          P\ar[r]^-{f'}\ar@{}[rd]|-{\square} \ar[d]_-{g'}     &         Q \ar@{->>}[d]^-{g}    \\
  M \ar[r]_-{f} & N         }\end{gathered}\end{equation}
If $g$ is a strict epimorphism, and $g'$ is an epimorphism, then the square is also cocartesian.
\end{prop}
\begin{proof}
See \cite[p. 10]{MR2571693}.
\end{proof}

By passing to the opposite category, we deduce from Proposition \ref{push-out}
the following result.
\begin{prop}\label{prop: cartesian}
Consider the following cocartesian square in $\mathcal C$.\begin{equation}\label{Equ: square cocartesian}\begin{gathered}\xymatrix{
          P\ar[r]^-{f'} \ar[d]_-{g'}     &         Q \ar@{->>}[d]^-{g}    \\
  M \ar[r]_-{f} & N         }\end{gathered}\end{equation}
If $g'$ is a strict monomorphism, and $g$ is a monomorphism, then the square is also cartesian.
\end{prop}
\section{Proto-abelian Category}

\subsection{Generalities}

\begin{defn}\label{proto}
We call \emph{proto-abelian category} a category $\mathcal C$ having a zero object, which satisfies the following conditions:
\begin{enumerate}[label=\rm(\alph*)]
\item \label{Axiom a}any morphism of $\mathcal{C}$ has a kernel and a cokernel.
\item \label{Axiom b}any morphism with zero kernel (resp. zero cokernel) is a monomorphism (resp. an epimorphism).
\item \label{Axiom c}the pull-back of a strict epimorphism by a strict monomorphism, which exists by \ref{Axiom a} and Proposition \ref{pull-back}, is a strict epimorphism; the push-forward of a strict monomorphism by a strict epimorphism, which exists by \ref{Axiom a} and Proposition \ref{Pro: push-forward ep}, is a strict monomorphism.
\end{enumerate}
Note that the opposite category of a proto-abelian category is also a proto-abelian category.
\end{defn}

\begin{rmk}
 The notion of proto-Abelian category could be considered as a non-additive analogue of the notion of abelian category, we refer to  \cite{MR2571693} for details where this notion has been proposed. See \cite[p. 21]{MR3970975}  for a more general notion of proto-exact category.
\end{rmk}

\begin{rmk}\label{Rem: pull-back which is an iso}
Let $\mathcal C$ be a proto-abelian category. Consider the following diagram
\[\xymatrix{&E\ar@{->>}[d]^-{\pi}\\F\ar@{ >->}[r]_-{f}&G}\]
of morphisms of $\mathcal C$, where $f$ is a strict monomorphism, and $\pi$ is a strict epimorphism. By Proposition \ref{pull-back}, the fiber product of $f$ and $\pi$ exists. In other words, one can complete the above diagram to a cartesian square.
\[\xymatrix{H\ar@{ >->}[r]^-g\ar@{->>}[d]_-{p}&E\ar@{->>}[d]^-{\pi}\\F\ar@{ >->}[r]_-{f}&G}
\]
Proposition \ref{pull-back} also shows that $g$ is a strict monomorphism.
Moreover, by axiom \ref{Axiom c}, we obtain that the morphism $p$ is a strict epimorphism, and in particular an epimorphism. By Proposition \ref{push-out}, the above square is not only cartesian but also cocartesian. In particular, if $g$ is an isomorphism, so is $f$.
\end{rmk}

\begin{egs}\label{examples}
\begin{enumerate}
\item Any Abelian category is a proto-abelian category (see \cite[p. 11]{MR2571693}).\\
\item \label{Item: spectral decomposition}Let $E$ be a Hermitian space, $\mathscr{A}$ be a Hermitian transform over $E$, $h$ be the minimal polynomial of $\mathscr{A}$ and $\mathcal C_h$ be a category defined as follows:

     An object of $\mathcal C_h$ is of the form $(V, \mathscr{B})$, where $V$ is a Hermitian space, $\mathscr{B}$ is a Hermitian transform over $V$ and $h(\mathscr{B})=0$. Let $(V, \mathscr{B})$ and $(V',\mathscr{B}')$ be two objects in $\mathcal C_h$, a morphism from $(V, \mathscr{B})$
     to $(V',\mathscr{B}')$ is a linear map $f:V\rightarrow V'$ such that the operator norm of $f \leq 1$ and $f\circ \mathscr{B}=\mathscr{B}'\circ f$.

     We call $\mathcal C_h$ be the category of Hermitian spaces with $h$ structure. Assume that $f: (V,\mathscr{B})\rightarrow (V',\mathscr{B}')$ is a morphism in $\mathcal{C}_h$. Note that
     $F=\{x\in V: f(x)=0\}$ is a $\mathscr{B}$-invariant subspace of $V$, $(F,\mathscr{B}|_F)$ is an object in $\mathcal C_h$ where the Hermitian norm is induced from $V$. In fact, the inclusion map $i: (F,\mathscr{B}|_F)\rightarrow (V,\mathscr{B})$ is a kernel of $f$. Denote $V'/f(V)$ by $G$. Since $f(V)$ is a $\mathscr{B}'$-invariant subspace of $V'$, $\mathscr{B}'$ induces a linear transform $\mathscr{B'}_G$ over $G$. Thus $(G,\mathscr{B}'_G)$ is an object in $\mathcal C_h$ where the Hermitian norm is the quotient norm induced from $V'$. In fact, the quotient map $\pi: (V',\mathscr{B}')\rightarrow (G,\mathscr{B}'_G)$ is a cokernel of $f$. Therefore, any morphism in $\mathcal C_h$ has a kernel and a cokernel.

     We claim that this category is a proto-abelian category. We only need to show \ref{Axiom c}. First, we show that the pull-back of a strict epimorphism by a strict monomorphism is a strict epimorphism. Let $f:(M,\mathscr{B}_M)\rightarrow (N,\mathscr{B}_N)$ be a strict monomorphism in $\mathcal C_h$, $g:(Q,\mathscr{B}_Q)\rightarrow (N,\mathscr{B}_N)$ be a strict epimorphism in $\mathcal C_h$, $\pi:(N,\mathscr{B}_N)\rightarrow (G,\mathscr{B}_G)$ be a cokernel of $f$ and $f':(P,\mathscr{B}_P)\rightarrow (Q,\mathscr{B}_Q)$ be a kernel of $\pi\circ g$. By Proposition \ref{pull-back}, we have the following Cartesian square.
       \begin{equation}
       \begin{gathered}\xymatrix{
          P\ar[r]^-{f'}\ar@{}[rd]|-{\square} \ar[d]_-{g'}     &         Q \ar@{->>}[d]^-{g}    \\
        M \ar@{>->}[r]_-{f} & N  \ar[r]_{\pi} &G       }\end{gathered}
       \end{equation}
    To show $g'$ is a strict epimorphism, we only need to show that the norm of $M$ is the same as the quotient norm induced from $P$. Let $y\in M$, we have
     \[\|y\|_M=\|f(y)\|_N=inf\{\|x\|_Q:g(x)=f(y)\}.\]
    For any $x\in Q$ satisfying $g(x)=f(y)$, we have $\pi\circ g(x)=\pi\circ f(y)=0$. Hence, there exists a unique $x'\in P$ satisfying $x=f'(x')$ and $\|x'\|_P=\|x\|_Q$ because $f'$ is a strict monomorphism by Proposition \ref{pull-back}. Since $f$ is a monomorphism, we have
    \[\|y\|_M=\inf\{\|x'\|_P:g\circ f'(x')=f(y)\}=\inf\{\|x'\|_P:g'(x')=y\}.\]
    After that, we need to show that the pushforward of a strict monomorphism by a strict epimorphism is a strict monomorphism. Let $f':(P,\mathscr{B}_P)\rightarrow (M,\mathscr{B}_M)$ be a strict monomorphism in $\mathcal C_h$, $g':(P,\mathscr{B}_P)\rightarrow (Q,\mathscr{B}_Q)$ be a strict epimorphism in $\mathcal C_h$, $i:(F,\mathscr{B}_F)\rightarrow (P,\mathscr{B}_P)$ be a kernel of $g'$ and $g:(M,\mathscr{B}_M)\rightarrow (N,\mathscr{B}_N)$ be a cokernel of $f'\circ i$. By Proposition \ref{Pro: push-forward ep}, we have the following cocartesian square.
    \begin{equation}
       \begin{gathered}\xymatrix{
       F\ar[r]^-{i}&   P\ar@{ >->}[d]_-{f'} \ar@{->>}[r]^-{g'}     &         Q \ar[d]^-{f}    \\
    &M \ar[r]_-{g} & N        }\end{gathered}
    \end{equation}
   To show $f$ is a strict monomorphism, we only need to show that the norm of $Q$ is the same as the induced norm from $N$. Let $y\in Q$, we have
   \[\|y\|_Q=\inf\{\|x\|_P:g'(x)=y\}=\inf\{\|f'(x)\|_M: f\circ g'(x)=f(y)\}\]
    \[=\inf\{\|f'(x)\|_M:g\circ f'(x)=f(y)\}=\|f(y)\|_N.\]
\item \label{item: Counterexample slope function}
  We consider a category $\mathcal{C}$ of four distinct objects $\boldsymbol{0},L,M,N$, such that $\boldsymbol{0}$ is the zero object and $\mathcal{C}$ only contains two non-zero and non-identity morphisms $f:L\rightarrow M$ and $g:M\rightarrow N$. Clearly one has $g\circ f=0$.
    
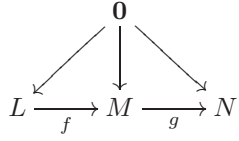
\[\xymatrix{
                &\boldsymbol{ 0} \ar[dl]\ar[d]\ar[dr] &            \\
   L  \ar[r]_-f &M\ar[r]_-g &N             }\]
  We can see that $f$ is the kernel of $g$ and $g$ is the cokernel of $f$.  We can further check that $\mathcal{C}$ is a proto-abelian category.
\item We consider the category $\mathcal C$ of finite dimensional normed linear spaces over a complete normed field $k$, where the morphisms are linear maps having operator norm $\leq 1$. Let $f:V\rightarrow W$ be a morphism in $\mathcal C$. Assume that $F=\{x\in V: f(x)=0\}$, $i:F\rightarrow V$ is the inclusion map and the norm of $F$ is the induced norm from $V$. Thus $i$ is a kernel of $f$. Assume that $G=W/f(V)$, $\pi: W\rightarrow G$ is the quotient map and the norm of $G$ is the quotient norm induced from $W$. Thus $\pi$ is a cokernel of $f$. Therefore, any morphism of $\mathcal C$ has a kernel and a cokernel. Furthermore, one can check that $\mathcal C$ is a proto-abelian category.
\end{enumerate}
\end{egs}

\begin{prop}\label{composite}
Let $\mathcal{C}$ be a proto-abelian category, and let $f: L\rightarrow M$ and $g: M\rightarrow N$ be morphisms of $\mathcal{C}$.
\begin{enumerate}[label=\rm(\arabic*)]
\item\label{Item: composition epimorphism} If $f$ and $g$ are both strict epimorphisms, so is $g\circ f$.
\item\label{Item: composition monomorphism} If $f$ and $g$ are both strict monomorphisms, so is $g\circ f$.
\end{enumerate}
\end{prop}
\begin{proof}
See \cite[p. 12]{MR2571693}.

\end{proof}

\begin{prop}\label{epimonic}
Let $f: M\rightarrow N$ be a morphism in a proto-abelian category $\mathcal{C}$, $f_0:N_1\rightarrow N$ be an image of $f$ and $f_1:M\rightarrow M_1$ be a coimage of $f$. Then there exists a unique morphism $g:M_1\rightarrow N_1$ such that  $f=f_0\circ g\circ f_1$. Moreover $g$ is both  an epimorphism and a monomorphism.
\end{prop}
\begin{proof}
See \cite[p. 12]{MR2571693}.
\end{proof}

\subsection{Sum of sub-objects}

\begin{nota}
Let $\mathcal C$ be a proto-abelian category. If $V\stackrel{f}{\rightarrow}E$ is a strict monomorphism of $\mathcal C$, we denote by $E/V$ an object of $\mathcal C$ which represents a cokernel of the morphism $f$. 
In other words, there is a universal morphism from $E$ to $E/V$  which is a cokernel of $f$.
\end{nota}

\begin{defn}\label{Def: sum of subobjects}
Let $C$ be a proto-abelian category, and \[\xymatrix{V\ar@{ >->}[r]^-{f}& E& W\ar@{ >->}[l]_-{g}}\] be a pair of strict monomorphism of $\mathcal{C}$. Let $p:E\rightarrow E/V$ be a cokernel of $f$, and $q:E\rightarrow E/W$ be a cokernel of $g$. By Proposition \ref{Pro: push-forward ep}, there exists an object $Q$ of $\mathcal C$ which represents the cofiber coproduct  of $p$ and $q$, and the universal morphisms $v:E/V\rightarrow Q$ and $w:E/W\rightarrow Q$ are strict epimorphisms.
\[\
\xymatrix{\relax V+W\ar@{ >->}[rd]|-{k}\\
          &E\ar@{->>}[r]^-{p} %\ar@{}[rd]|-{\square}
          \ar@{->>}[d]_-{q}
          &         E/V \ar@{->>}[d]^-{v}    \\
& E/W\ar@{->>}[r]_-w & Q         }\]
We denote by $V+W$ an object of $\mathcal C$ which represents the kernel of $v\circ p=w\circ q$.
\end{defn}

\begin{prop}\label{properties}We keep the notation and hypotheses of Definition \ref{Def: sum of subobjects}. Let $k:V+W\rightarrow E$ be the universal morphism.
\begin{enumerate}[label=\rm(\arabic*)]
\item\label{Item:factor through} There exists a unique morphism $u:V\rightarrow V+W$ (resp. $t:W\rightarrow V+W$), such that $f=k\circ u$ (resp. $g=k\circ t$). Moreover, both $u$ and $t$ are strict monomorphisms. In addition, the square
\[\xymatrix{V\times_EW\ar[r]\ar[d]&V\ar[d]^u\\ W\ar[r]_t&V+W}\]
is cartesian.
\item \label{Item:universal}Let $h:H\rightarrow E$ be a strict monomorphism of $\mathcal C$. If there exists $f':V\rightarrow H$ and $g':W\rightarrow H$ such that $f=h\circ f'$ and $g=h\circ g'$ hold, then there exists a unique morphism $s:V+W\rightarrow H$ such that $k=h\circ s$ and $s$ is a strict monomorphism.
\end{enumerate}
\end{prop}
\begin{proof}
\ref{Item:factor through} By symmetry it suffices to prove the statement for $f$. Since $v\circ p\circ f=0$ and $k$ is a kernel of $v\circ p$, there exists a unique morphism $u:V\rightarrow V+W$ such that $f=k\circ u$. By Proposition \ref{strict}, we obtain that $u$ is a strict monomorphism. For the same reason, there exists a unique morphism $t:W\rightarrow V+W$ such that $g=k\circ t$, and $t$ is a strict monomorphism.\\
\[\
\xymatrix{\relax V+W\ar@{ >->}[rd]^-{k}\\
          V\ar@{ >->}[r]^-{f}\ar[u]^-{u}&E\ar@{->>}[r]^-{p} %\ar@{}[rd]|-{\square}
          \ar@{->>}[d]_-{q}
          &         E/V \ar@{->>}[d]^-{v}    \\
& E/W\ar@{->>}[r]_-w & Q         }\]
Let $\alpha:X\rightarrow V$ and $\beta:X\rightarrow W$ be morphisms such that $u\circ \alpha=t\circ \beta$, we have $f\circ \alpha=k\circ u\circ \alpha=g\circ \beta$, then there exists a unique $\gamma:X\rightarrow V\times_E W$ such that $\alpha=t'\circ \gamma$ and  $\beta=u'\circ \gamma$.
This means that the following diagram is cartesian.
\[\xymatrix{V\times_EW\ar^{t'}[r]\ar[d]_{u'}&V\ar[d]^u\\ W\ar[r]_t&V+W}\]
\ref{Item:universal} Let $r:E\rightarrow E/H$ be a cokernel of $h$. Since $f=h\circ f'$ (resp: $g=h\circ g'$), one has $r\circ f=0$ (resp. $r\circ g=0$). Then there exists a unique morphism $p':E/V\rightarrow E/H$ (resp. $q':E/W\rightarrow E/H$) such that $r=p'\circ p$ (resp. $r=q'\circ q$). Therefore, there exists a unique morphism $\delta: Q\rightarrow E/H$ such that $p'=\delta\circ v$ and $q'=\delta\circ w$ by the universal property of the push-out square. Hence, $r\circ k=0$.  Note that  the strict monomorphism $h$ is a kernel of $r$ by Lemma \ref{Lem: kernel cokernel}. Then there exists a unique morphism $s:V+W\rightarrow H$ such that $k=h\circ s$ because of the universal property of kernel. By Proposition \ref{strict}, $s$ is a strict monomorphism.
\[\xymatrix{&W\ar@{ >->}[d]_{g'}\ar@/^/[rdd]^-g\\
V\ar@{ >->}[r]^{f'}\ar@/_/[rrd]_f&H\ar@{ >->}[rd]|-h\\
&&E\ar@{->>}[rd]_-r\ar@{->>}[dd]_-q\ar@{->>}[rr]^-p&&E/V\ar@{->>}[dd]^-v\ar[ld]_{p'}\\&&&E/H\\
&&E/W\ar[ru]^{q'}\ar@{->>}[rr]_-w&&Q\ar[lu]_{\delta}}\]

\end{proof}

\begin{rmk}\label{Rem: cocartesian somme}
If the cofibre coproduct $P$ of the universal morphisms $V\leftarrow V\times_EW\rightarrow W$ exists, and the unique morphism $\lambda:P\rightarrow E$ given by the universal property of cofibre coproduct is a strict monomorphism, 
then $P$ is canonically isomorphic to $V+W$.
\end{rmk}

\section{Slope Functions}\label{Sec: Slope Functions}

In this section, we fix a small proto-abelian category $\mathcal C$.

\subsection{Slope function}
\begin{defn}
Let $E$ be an object of $\mathcal C$. We denote by $\mathcal S_E$  the category defined as follows: The objects of $\mathcal S_E$ are diagrams of the form
\[\xymatrix{W'\ar@{ >->}[r]^-f&W\ar@{ >->}[r]^-g&E},\]
where $f$ and $g$ are strict monomorphisms.  If \[\mathcal W_1=(\xymatrix@C-.3pc{W'_1\ar@{ >->}[r]^-{f_1}&W_1\ar@{ >->}[r]^-{g_1}&E})\quad\text{ and  }\quad \mathcal W_2=(\xymatrix@C-.3pc{W'_2\ar@{ >->}[r]^-{f_2}&W_2\ar@{ >->}[r]^-{g_2}&E})\]
are two objects of $\mathcal S_E$, the morphisms from $\mathcal W_1$ to $\mathcal W_2$  are couples of morphisms $(\alpha,\beta)$ such that the following diagram commutes.

\[\xymatrix{W'_1\ar@{ >->}[r]^-{f_1}\ar[d]_-\alpha&W_1\ar@{ >->}[r]^-{g_1}\ar[d]_-{\beta}&E\ar[d]^-{\operatorname{id}_E}\\
W'_2\ar@{ >->}[r]_-{f_2}&W_2\ar@{ >->}[r]_-{g_2}&E}\]
By Proposition \ref{strict},
both $\beta$ and $\alpha$ are strict monomorphisms. Since $g_2$ and $g_2\circ f_2$ are monomorphism, $(\alpha,\beta)$ is unique.
\end{defn}

Note that $(\alpha,\beta )$ is an isomorphism in $\mathcal S_E$ if and only if both morphisms $\alpha$ and $\beta$ are isomorphisms in $\mathcal C$. We denote by $\mathcal S_E^0$ the set of objects
\[\xymatrix{W'\ar@{ >->}[r]^-f&W\ar@{ >->}[r]^-g&E},\]
in $\mathcal S_E$ such that the morphism $f$ is \emph{not} an isomorphism.

\begin{defn}\label{Def: slope function}
Let $\Lambda$ be a  partially ordered set. We call \emph{slope function of $E$ valued in $\Lambda$} any map $\mu:\mathcal S_E^0\rightarrow\Lambda$ which satisfies the following
\emph{slope inequality}: for any couple \[\mathcal W_1=(\xymatrix@C-.3pc{W'_1\ar@{ >->}[r]^-{f_1}&W_1\ar@{ >->}[r]^-{g_1}&E})\quad\text{ and  }\quad \mathcal W_2=(\xymatrix@C-.3pc{W'_2\ar@{ >->}[r]^-{f_2}&W_2\ar@{ >->}[r]^-{g_2}&E})\] \label{slope}
of elements in $\mathcal S_E^0$, and any morphism $(\alpha,\beta)$ from $\mathcal W_1$ to $\mathcal W_2$ in the category $\mathcal S_E$, if the diagram
\[\xymatrix{W'_1\ar@{ >->}[r]^-{f_1}\ar@{ >->}[d]_-\alpha&W_1\ar@{ >->}[d]^-{\beta}\\
W'_2\ar@{ >->}[r]_-{f_2}&W_2}\]
is cartesian and the induced strict monomorphism $u:W_1+W_2'\rightarrow W_2$ is an isomorphism, then $\mu(\mathcal W_1)\leqslant\mu(\mathcal W_2)$.
\end{defn}

Note that this slope inequality implies that, if $(\alpha,\beta)$ is an isomorphism from $\mathcal W_1$ to $\mathcal W_2$ in the category $\mathcal S_E$, then $\mu(\mathcal W_1)=\mu(\mathcal W_2)$. In fact, if $\alpha$ and $\beta$ are isomorphisms, then both diagrams
\[\begin{gathered}\xymatrix{W'_1\ar@{ >->}[r]^-{f_1}\ar@{ >->}[d]_-\alpha&W_1\ar@{ >->}[d]^-{\beta}\\
W'_2\ar@{ >->}[r]_-{f_2}&W_2}\end{gathered}\quad\text{ and }\quad\begin{gathered}\xymatrix{W'_2\ar@{ >->}[r]^-{f_2}\ar@{ >->}[d]_-{\alpha^{-1}}&W_2\ar@{ >->}[d]^-{\beta^{-1}}\\
W'_1\ar@{ >->}[r]_-{f_1}&W_1}\end{gathered}\]
are cartesian. Moreover, the induced morphisms $W_1+W_2'\rightarrow W_2$ and $W_2+W_1'\rightarrow W_1$ are isomorphisms. Therefore, we obtain $\mu(\mathcal W_1)\leqslant\mu(\mathcal W_2)$ and $\mu(\mathcal W_2)\leqslant\mu(\mathcal W_1)$, namely $\mu(\mathcal W_1)=\mu(\mathcal W_2)$.

\begin{rmk}
Let \[\mathcal W_1=(\xymatrix@C-.3pc{W'_1\ar@{ >->}[r]^-{f_1}&W_1\ar@{ >->}[r]^-{g_1}&E})\quad\text{ and  }\quad \mathcal W_2=(\xymatrix@C-.3pc{W'_2\ar@{ >->}[r]^-{f_2}&W_2\ar@{ >->}[r]^-{g_2}&E})\] \label{slope}
 be elements in $\mathcal S_E^0$, $(\alpha,\beta):\mathcal W_1\rightarrow \mathcal W_2$ be a morphism in the category $\mathcal S_E$.
  Suppose the diagram below is cartesian and the induced morphism $\nu:W_1/W_1'\rightarrow W_2/W_2'$ is an isomorphism.
\[\xymatrix{W'_1\ar@{ >->}[r]^-{f_1}\ar@{ >->}[d]_-\alpha&W_1\ar@{ >->}[d]^-{\beta}\\
W'_2\ar@{ >->}[r]_-{f_2}&W_2}\]
Then the universal morphism $u:W'_2+W_1\rightarrow W_2$ is an isomorphism.
\end{rmk}

\begin{nota}\label{Not: mu of subquotient}
Let
$(W'\stackrel{f}{\rightarrowtail}W\stackrel{g}{\rightarrowtail}E)$
 be an element in $\mathcal S_E^0$,
 if there is no ambiguity on the morphisms $f$ and $g$, we denote by $\mu(W/W')$ the element \[\mu(\xymatrix{W'\ar@{ >->}[r]^-{f}&W\ar@{ >->}[r]^-{g}&E})\in\Lambda.\] If in addition $W'$ is a zero object, then $\mu(W/W')$ is also denoted by $\mu(W)$ for simplicity.
\end{nota}

\begin{defn}\label{Def: strong slope inequality}
Let $E$ be an object in $\mathcal C$, $\Lambda$ be a  partially ordered set and $\mu:\mathcal S_E^0\rightarrow\Lambda$ be a map. We say that $\mu$ \emph{satisfies the strong slope inequality} if for any couple \[\mathcal W_1=(\xymatrix@C-.3pc{W'_1\ar@{ >->}[r]^-{f_1}&W_1\ar@{ >->}[r]^-{g_1}&E})\quad\text{ and  }\quad \mathcal W_2=(\xymatrix@C-.3pc{W'_2\ar@{ >->}[r]^-{f_2}&W_2\ar@{ >->}[r]^-{g_2}&E})\]  of elements in $\mathcal S_E^0$, and any morphism $(\alpha,\beta)$ from $\mathcal W_1$ to $\mathcal W_2$ in the category $\mathcal S_E$, if the diagram \[\xymatrix{W'_1\ar@{ >->}[r]^-{f_1}\ar@{ >->}[d]_-\alpha&W_1\ar@{ >->}[d]^-{\beta}\\
W'_2\ar@{ >->}[r]_-{f_2}&W_2}\]
is cartesian, then the inequality $\mu(\mathcal W_1)\leqslant\mu(\mathcal W_2)$ always holds. Note that, if $\mu$ satisfies the strong slope inequality, then it is necessarily a slope function of $E$ valued in $\Lambda$.
\end{defn}

\begin{defn}
Let $(\Lambda,\leqslant)$ be a partially ordered set and $A$ be a subset of $\Lambda$. We say that $A$ admits a \emph{supremum} if there is an element of $\Lambda$, denoted by $\sup(A)$, which satisfies the following conditions:
\begin{enumerate}[label=\rm(\arabic*)]
\item $\sup(A)$ is an upper bound of $A$, namely for any $a\in A$ one has $a\leqslant\sup(A)$.
\item $\sup(A)$ is the least upper bound of $A$, namely any upper bound $b$ of $A$ satisfies $\sup(A)\leqslant b$.
\end{enumerate}
\end{defn}

\begin{prop}\label{prop: mu max}
We keep the notation of Definition \ref{Def: slope function}. Assume that any non-empty subset of $\Lambda$ admits a supremum. Let $\mu_{\max}:\mathcal S_E^0\rightarrow\Lambda$ be the map sending
$(\xymatrix@C-.3pc{W'\ar@{ >->}[r]^-{f}&W\ar@{ >->}[r]^-{g}&E})\in\mathcal S_E^0$
to 
\begin{equation*}
\begin{aligned}
&\sup\big\{\mu(\xymatrix@C-.3pc{W'\ar@{ >->}[r]^-{\widetilde f}&\widetilde W\ar@{ >->}[r]^-{\widetilde g}&E})\,:\,(\xymatrix@C-.3pc{W'\ar@ { >->}[r]^-{\widetilde f}&\widetilde W\ar@{ >->}[r]^-{\widetilde g}&E})\in \mathcal S_E^0,\\&
\text{\and there exists $h:\widetilde W\rightarrow W$ such that
$f=h\circ \widetilde f$ and $\widetilde g=g\circ h$
}\big\}.\\
\end{aligned}
\end{equation*}
 Then $\mu_{\max}$ is a slope function of $E$ which satisfies the strong slope inequality.
\end{prop}
\begin{proof}
Let \[\mathcal W_1=(\xymatrix@C-.3pc{W'_1\ar@{ >->}[r]^-{f_1}&W_1\ar@{ >->}[r]^-{g_1}&E})\quad\text{ and  }\quad \mathcal W_2=(\xymatrix@C-.3pc{W'_2\ar@{ >->}[r]^-{f_2}&W_2\ar@{ >->}[r]^-{g_2}&E})\]  be elements in $\mathcal S_E^0$, and $(\alpha,\beta)$ be a morphism from $\mathcal W_1$ to $\mathcal W_2$ in the category $\mathcal S_E$, such that the diagram \[\xymatrix{W'_1\ar@{ >->}[r]^-{f_1}\ar@{ >->}[d]_-\alpha&W_1\ar@{ >->}[d]^-{\beta}\\
W'_2\ar@{ >->}[r]_-{f_2}&W_2}\]
is cartesian. Let
\[\xymatrix{W_1'\ar@{ >->}[r]^{\varphi_1}&\widetilde W_1\ar@{ >->}[r]^-{h_1}&W_1}\]
be a decomposition of $f_1$, where $\varphi_1$ and $h_1$ are strict monomorphisms. Let $\widetilde W_2=\widetilde W_1+W_2'$, $k:\widetilde W_2=\widetilde W_1+W_2'\longrightarrow E$ be the universal morphism and $\gamma:\widetilde W_1\rightarrow\widetilde W_2$ be the unique morphism such that $\gamma\circ k=g_1\circ h_1$ (see Proposition \ref{properties}). By Proposition \ref{properties}, the diagram
\[\xymatrix{W_1'\ar@{ >->}[d]_\alpha\ar@{ >->}[r]^-{\varphi_1}&\widetilde W_1\ar@{ >->}[d]^-{\gamma}\\
W_2'\ar@{ >->}[r]_{\varphi_2}&\widetilde{W}_2}\]
is cartesian.  Therefore, the slope inequality leads to
\[\mu(\xymatrix@C-.3pc{W'_1\ar@{ >->}[r]^-{\varphi_1}&\widetilde W_1\ar@{ >->}[r]^-{g_1\circ h_1}&E})\leqslant\mu(\xymatrix@C-.3pc{W'_2\ar@{ >->}[r]^-{\varphi_2}&\widetilde W_2\ar@{ >->}[r]^-{k}&E})\leqslant\mu_{\max}(\mathcal W_2).\]
By taking the supremum with respect to the decompositions of $f_1$, we obtain that $\mu_{\max}(\mathcal W_1)\leqslant\mu_{\max}(\mathcal W_2)$.
\end{proof}

\begin{rmk}
In Propsition \ref{prop: mu max}, if $\mu:\mathcal S_E^0\rightarrow\Lambda$ satisfies the strong slope inequality, then $\mu_{\max}=\mu$.
\end{rmk}

\subsection{Induced slope functions}

Let $\mathcal C$ be a  small  proto-abelian category, $E$ be an object of $\mathcal C$, $\Lambda$ be a partially ordered set and $\mu:\mathcal S_E^0\rightarrow\Lambda$ be a slope function of $E$ valued in $\Lambda$.\\

Let $\eta:F\rightarrow E$ be a strict monomorphism of  $\mathcal C$. We define a function $\mu_\eta:\mathcal S_F^0\rightarrow\Lambda$ as follows.
For any
\[\mathcal Y=(\xymatrix{Y'\ar@{ >->}[r]^-{u}&Y\ar@{ >->}[r]^-{v}&F})\in\mathcal S_F^0,\]
we let
\[\mu_\eta(\mathcal Y):=\mu(\xymatrix{Y'\ar@{ >->}[r]^-{u}&Y\ar@{ >->}[r]^-{\eta\circ v}&E}).\]
This is actually a slope function of $F$ valued in $\Lambda$. If $\mu$ satisfies the strong slope inequality, so does $\mu_{\eta}$.
If $\eta_1:F_1\rightarrow E$ and $\eta_2:F_2\rightarrow F_1$ are strict monomorphisms of $\mathcal C$, then $\eta_1\circ\eta_2$ is a strict monomorphism of $\mathcal C$ (by Proposition \ref{composite}), and one has
\[\mu_{\eta_1\circ\eta_2}=(\mu_{\eta_1})_{\eta_2}.\]

Let $\pi:E\rightarrow G$ be a strict epimorphism. We define a function $\mu^\pi:\mathcal S_{G}^0\rightarrow\Lambda$ as follows. For any object
 \[\mathcal Z=(\xymatrix{Z'\ar@{ >->}[r]^-{u}&Z\ar@{ >->}[r]^-{v}&G})\]
of $\mathcal S_{G}$, we denote by $\mathcal Z\times_{G}E$ an arbitrary object
\[\xymatrix{W'\ar@{ >->}[r]^-f&W\ar@{ >->}[r]^-g&E}\] of $\mathcal S_E$ which fits into a diagram
\[\xymatrix{W'\ar@{}[rd]|-{\square}\ar@{ >->}[r]^-f\ar@{->>}[d]&W\ar@{}[rd]|-{\square}\ar@{ >->}[r]^-g\ar@{->>}[d]&E\ar@{->>}[d]^-{\pi}\\
Z'\ar@{ >->}[r]_-{u}&Z\ar@{ >->}[r]_-{v}&G}\]
in which all squares are cartesian. If $\mathcal Z$ belongs to $\mathcal S_{G}^0$, then necessarily $\mathcal Z\times_GE$ belongs to $\mathcal S_E^0$ (see Remark \ref{Rem: pull-back which is an iso}). We then define \[\mu^\pi(\mathcal Z):=\mu(\mathcal Z\times_GE)=\mu(W/W')\in\Lambda.\]
Note that the value of $\mu^\pi(\mathcal Z)$ does not depend on the choice of the object $\mathcal Z\times_GE$, thanks to the slope inequality in  Definition \ref{Def: slope function}.

\begin{prop}\label{prop: quotient solpe}
 The function $\mu^\pi:\mathcal S_G^0\rightarrow\Lambda$ is a slope function of $G$ valued in $\Lambda$. Moreover, if $\mu$ satisfies the strong slope inequality, so does $\mu^\pi$.
\end{prop}
\begin{proof}
It suffices to check that $\mu^\pi$ satisfies the slope inequality.  Let
\[\mathcal Z_1=(\xymatrix@C-.3pc{Z_1'\ar@{ >->}[r]^-{u_1}&Z_1\ar@{ >->}[r]^-{v_1}&G})\quad\text{ and  }\quad \mathcal Z_2=(\xymatrix@C-.3pc{Z'_2\ar@{ >->}[r]^-{u_2}&Z_2\ar@{ >->}[r]^-{v_2}&G})\]
be elements of $\mathcal S_G^0$, $(\alpha,\beta)$ be a morphism from $\mathcal Z_1$ to $\mathcal Z_2$ in the category $\mathcal S_G$, such that the diagram
\begin{equation}\label{Equ: square u1u2}\begin{gathered}\xymatrix{Z_1'\ar@{}[rd]|-{\square}\ar@{ >->}[r]^{u_1}\ar@{ >->}_{\alpha}[d]&Z_1\ar@{ >->}^{\beta}[d]\\
Z'_2\ar@{ >->}_{u_2}[r]&Z_2}\end{gathered}\end{equation}
is cartesian and induces an isomorphism from $Z_1+Z_2'$ to $Z_2$.

Let \[\mathcal W_1=\mathcal Z_1\times_E G=(\xymatrix@C-.3pc{W_1'\ar@{ >->}[r]^-{f_1}&W_1\ar@{ >->}[r]^-{g_1}&E}),\mathcal W_2=\mathcal Z_2\times_E G=(\xymatrix@C-.3pc{W'_2\ar@{ >->}[r]^-{f_2}&W_2\ar@{ >->}[r]^-{g_2}&E}),\] where $W_1'= Z_1'\times_G E$,  $W_1=Z_1\times_G E$, $W_2'=Z'_2\times_G E$ and $W_2=Z_2\times_G E$. Consider the following cubic commutative diagram below,
\begin{equation}\label{Equ: cube diagram}\begin{gathered}\xymatrix@C+,2pc@R+,2pc@!0{\relax
  & W_1' \ar@{ >->}[rr]\ar@{>->}[dl] \ar@{->>}'[d][dd]
     &  &  W_1\ar@{->>}[dd] \ar@{>->}[dl]       \\
   \ar@{ >->}[rr]\ar@{->>}[dd]
   W'_2   & &  W_2\ar@{->>}[dd] &\\
  & Z_1' \ar@{ >->}'[r][rr]\ar@{>->}[dl]
      &  & Z_1\ar@{>->}[dl]                \\
  Z'_2 \ar@{ >->}[rr]
      & & Z_2  &      }\end{gathered}\end{equation}
where all vertical arrows and arrows in the upper square are universal morphisms.
The morphisms in the upper square are strict monomorphisms because of Proposition \ref{strict}, and the vertical morphisms are strict epimorphisms because of the Axiom\ref{Axiom c} of  proto-abelian category. Moreover, since the pull-back preserves projective limits (see \cite{MR2182076} Proposition 2.1.7), we obtain that
the diagram
\begin{equation}\label{Equ:W catesian}\begin{gathered}\xymatrix{W_1'\ar@{ >->}[r]^{f_1}\ar@{ >->}_{a}[d]&W_1\ar@{ >->}^{b}[d]\\
W'_2\ar@{ >->}_{f_2}[r]&W_2}\end{gathered}
\end{equation}
forming the upper square of the above cubic diagram is cartesian. Now, we want to show $W_2$ is canonically isomorphic to $W_1+W_2'$. Assume that $h:H\rightarrow E$,  $s:W_1\rightarrow H$, $t:W_2'\rightarrow H$ are strict monomorphisms where both $g_1=h\circ s$ and $g_2\circ f_2=h\circ t$ hold.
Note that $p_1:W_1\rightarrow Z_1$ is a strict epimorphism, by Proposition \ref{Pro: push-forward ep}, there exists a strict epimorphism $\pi':H\rightarrow F$ and a strict monomorphism $\varphi:Z_1\rightarrow F$ making the following diagram cocartesian.
\[\xymatrix{W_1\ar@{ >->}[r]^{s}\ar@{ ->>}_{p_1}[d]&H\ar@{ ->>}^{\pi'}[d]\\
	Z_1\ar@{ >->}_{\varphi}[r]&F}\] Since the following diagram
\[\xymatrix{W_1\ar@{ >->}[r]^{g_1}\ar@{ ->>}_{p_1}[d]&E\ar@{ ->>}^{\pi}[d]\\
	Z_1\ar@{ >->}_{v_1}[r]&G}\]
is also cocartesian, by the universal property of cocartesian diagram, there exists a unique morphism $\lambda: F\rightarrow G$ such that $v_1=\lambda\circ \varphi$ and all squares in the following diagram are cocartesian.
\[\xymatrix{W_1\ar@{ >->}[r]^-{s}\ar@{ ->>}[d]_-{p_1}& H\ar@{ >->}[r]^{h}\ar@{ ->>}_-{\pi'}[d]&E\ar@{ ->>}^{\pi}[d]\\
	Z_1\ar@{ >->}[r]_-{\varphi} & F\ar@{ >->}_{\lambda}[r]&G}\]Because $h$ is a strct monomorphism, one has that $\lambda$ is a strict monomorphism(See Definition \ref{proto} \ref{Axiom c}). Therefore, all squares in the above diagram are cartesian (See Proposition \ref{prop: cartesian}). Assume that $q_0:W_0\rightarrow W_2'$ is a kernel of $q_2:W_2'\rightarrow Z_2'$. Notice that \[\lambda\circ \pi'\circ t\circ q_0=\pi\circ g_2\circ f_2\circ q_0=v_2\circ u_2\circ q_2\circ q_0=0,\] $\pi'\circ t\circ q_0=0$ since $\lambda$ is a monomorphism. Since $q_2$ is a cokernel of $q_0$ (See Lemma \ref{Lem: kernel cokernel}), there exists a unique morphism $\psi: Z_2'\rightarrow F$ such that $\pi'\circ t=\psi\circ q_2$. Notice that $\lambda\circ\psi\circ q_2=v_2\circ u_2\circ q_2$,  $v_2\circ u_2=\lambda\circ \psi$ since $q_2$ is an epimorphism. By Proposition \ref{strict}, $\psi$ is a strict monomorphism.
 \[\xymatrix{W_0\ar[d]_-{q_0}&&\\
		W_2'\ar@{ >->}[r]^-{t}\ar@{ ->>}[d]_-{q_2}& H\ar@{ >->}[r]^{h}\ar@{ ->>}_-{\pi'}[d]&E\ar@{ ->>}^{\pi}[d]\\
	Z_2'\ar@{ >->}[r]_-{\psi} & F\ar@{ >->}_{\lambda}[r]&G}\]
Notice that \[\lambda\circ \varphi\circ u_1=v_1\circ u_1=v_2\circ u_2\circ \alpha=\lambda\circ \psi\circ \alpha,\]one has that $\varphi\circ u_1=\psi\circ \alpha$ since $\lambda$ is a monomorphism.
By Proposition \ref{properties} \ref{Item:universal}, there exists a unique morphism $\gamma:Z_2\rightarrow F$ such that $v_2=\lambda\circ\gamma$.
\[
		\xymatrix{
		  Z_1' \ar[r]^-{u_1} \ar[d]_-{\alpha}     &         Z_1 \ar[d]^-{\beta}  \ar@/^/[ddr]^-{\varphi} & \\
		  Z_2' \ar[r]_-{u_2} \ar@/_/[drr]_-{\psi} & Z_2\ar[dr]^-{\gamma}     &\\
		  & & F   }\]
Notice that the following diagrams \[\begin{gathered}\xymatrix{W_2\ar@{}[rd]|-{\square}\ar@{ >->}[r]^{g_2}\ar@{ ->>}_{p_2}[d]&E\ar@{ ->>}^{\pi}[d]\\
	Z_2\ar@{ >->}_{v_2}[r]&G}\end{gathered} \quad\text{ and }\quad
	\begin{gathered}\xymatrix{H\ar@{}[rd]|-{\square}\ar@{ >->}[r]^{h}\ar@{ ->>}_{\pi'}[d]&E\ar@{ ->>}^{\pi}[d]\\
	F\ar@{ >->}_{\lambda}[r]&G}\end{gathered}\]
are cartesian, there exists a unique morphism $k:W_2\rightarrow H$ such that $g_2=h\circ k$ and all squares in the following diagram are cartesian.
 \[\xymatrix{W_2\ar@{}[rd]|-{\square}\ar@{ >->}[r]^-{k}\ar@{ ->>}[d]_-{p_2}& H\ar@{}[rd]|-{\square}\ar@{ >->}[r]^{h}\ar@{ ->>}_-{\pi'}[d]&E\ar@{ ->>}^{\pi}[d]\\
	Z_2\ar@{ >->}[r]_-{\gamma} & F\ar@{ >->}_{\lambda}[r]&G}\]Notice that
\[h\circ t=g_2\circ f_2=h\circ k\circ f_2, \quad h\circ s=g_1=g_2\circ b=h\circ k\circ b.\]
 Therefore, $t=k\circ f_2$ and $s=k\circ b$ since $h$ is a monomorphism. Since $H$ is arbitrary, take $H=W_2'+W_1$ and in this situation $k$ is a strict monomorphism from $W_2$ to $W_1+W_2'$ such that $g_2=h\circ k$. Moreover, by Proposition \ref{properties} \ref{Item:universal}, there exists a unique morphism $k':W_1+W_2'\rightarrow W_2$ such that $h=g_2\circ k'$. One can check that $k$ and $k'$ are inverse to each other. 

Since $\mu$ is a slope function, we have $\mu(W_1/W_1')\leqslant\mu(W_2/W_2')$, namely $\mu^\pi(\mathcal Z_1)\leqslant\mu^\pi(\mathcal Z_2)$. Hence $\mu^\pi$ is also a slope function. The same argument shows that, if $\mu$ satisfies the strong slope inequality, so does $\mu^{\pi}$.
\end{proof}

If $\pi_1:E\rightarrow G_1$ and $\pi_2:G_1\rightarrow G_2$ are strict epimorphisms of $\mathcal C$, by Proposition \ref{composite} the composed morphism $\pi_2\circ\pi_1:E\rightarrow G_2$ is also a strict epimorphism. Moreover, one has $\mu^{\pi_2\circ\pi_1}=(\mu^{\pi_1})^{\pi_2}$.

\begin{prop}\label{prop:compactilblity}
Let
\[\xymatrix{F\ar@{}[rd]|-{\square}\ar@{ >->}[r]^-{\eta}\ar@{->>}_-{\pi}[d]&E\ar@{->>}^-{\pi'}[d]\\
H\ar@{ >->}_-{\eta'}[r]&G}\]
be a commutative diagram of  morphisms of $\mathcal C$. Assume that the horizontal arrows of the diagram are strict monomorphisms, the vertical arrows are strict epimorphisms and that the square is cartesian. If $\mu:\mathcal S_E^0\rightarrow\Lambda$ is a slope function, then
the following equality holds:
\[(\mu_\eta)^\pi=(\mu^{\pi'})_{\eta'}.\]
\end{prop}
\begin{proof}
Assume that \[\mathcal Y=(\xymatrix{Y'\ar@{ >->}[r]^-{u}&Y\ar@{ >->}[r]^-{v}&H})\in\mathcal S_H^0.\]
Let
$\mathcal Z=\mathcal Y\times_{H} F$, and thus it is an element in $\mathcal S_F^0$. We denote $\mathcal Z$ by \[(\xymatrix{Z'\ar@{ >->}[r]^-{u}&Z\ar@{ >->}[r]^-{v}&F})\] and it fits in the diagram below,
\[\xymatrix{Z'\ar@{}[rd]|-{\square}\ar@{ >->}[r]^-f\ar@{->>}[d]&Z\ar@{}[rd]|-{\square}\ar@{ >->}[r]^-g\ar@{->>}[d]&F\ar@{}[rd]|-{\square}\ar@{ >->}[r]^-{\eta}\ar@{->>}[d]_-{\pi}&E\ar@{->>}[d]^-{\pi'}\\
Y'\ar@{ >->}[r]_-{u}&Y\ar@{ >->}[r]_-{v}&H\ar@{ >->}[r]_-{\eta'}&G}\]
where each square is cartesian. By Proposition \ref{pull-back}, both $f$ and $g$ are strict monomorphism.
Hence,
\[(\mu_\eta)^\pi(\mathcal{Y})=\mu_\eta(Z/Z')=\mu(Z/Z')\]and
\[(\mu^{\pi'})_{\eta'}(\mathcal Y)=\mu^{\pi'}(Y/Y')=\mu(Z/Z').\]
Therefore,
 \[(\mu_\eta)^\pi(\mathcal Y)=(\mu^{\pi'})_{\eta'}(\mathcal Y).\]

\end{proof}

\begin{rmk}\label{Rem: induced slope}
The above compatibility result shows that a slope function $\mu$ of $E$ induces in a canonical way, for each strict subquotient $H$ of $E$, a slope function of $H$. By abuse of notation, in the case where there is no ambiguity on the underlying subquotient diagram of morphisms linking $H$ and $E$, we often denote by $\mu$ the induced slope function of $H$ for simplicity. Note that this convention fits well with Notation \ref{Not: mu of subquotient}.
\end{rmk}

\subsection{Slope functions over bounded lattices}
In this subsection, we give an alternative definition of slope functions which is more general. We fix a partially ordered set having at least two elements $(\Gamma,\leqslant)$. The following definition comes from order theory, and to make the notation look consistent with the notation of proto-abelian category, we use $E$ to represent maximal element in $\Gamma$ instead of $1$.
 \begin{defn}
 Let $E_1,E_2$ be elements in $\Gamma$. An element $A\in \Gamma$ is called the infimum of $E_1,E_2$ in $\Gamma$ if it satisfies the following conditions:
 \begin{enumerate}
 \item $A\leqslant E_1$ and $A\leqslant E_2$.
 \item For any $B\in \Gamma$ satisfying both $B\leqslant E_1$ and $B\leqslant E_2$, we have $B\leqslant A$.
 \end{enumerate}
 If the infimum of $E_1,E_2$ exists, it is unique. We denote it by $E_1\wedge E_2$.\\
 An element $A\in \Gamma$ is called the supremum of $E_1,E_2$ in $\Gamma$ if it satisfies the following conditions:
 \begin{enumerate}
 \item $A\geqslant E_1$ and $A\geqslant E_2$.
 \item For any $B\in \Gamma$ satisfying both $B\geqslant E_1$ and $B\geqslant E_2$, we have $B\geqslant A$.
 \end{enumerate}
 If the supremum of $E_1,E_2$ exists, it is unique. We denote it by $E_1\vee E_2$.\\
 We say $(\Gamma,\leqslant)$ is a bounded lattice if it satisfies the following conditions:
\begin{enumerate}
\item there exists a largest element $E$ in $\Gamma$, such that for any $A\in \Gamma$, $A\leqslant E$.
\item there exists a smallest element $\boldsymbol{0}$ in $\Gamma$, such that for any $A\in \Gamma$, $\boldsymbol{0}\leqslant A$ and  $\boldsymbol{0}\neq E$.
 \item
 For any $E_1,E_2\in \Gamma$,  both $E_1\wedge E_2$ and $E_1\vee E_2$ exist in $\Gamma$.
\end{enumerate}

 \end{defn}

\begin{defn}
Let $E$ be an object of a proto-abelian category $\mathcal C$. We call \emph{strict sub-object} of $E$ any strict monomorphism in the category $\mathcal C$, whose target object identifies with $E$. If a strict sub-object of $E$ is not an isomorphism, we say that it is \emph{proper}.

If $f:V\rightarrow E$ and $g:W\rightarrow E$ are two strict sub-objects of $E$, we call morphism from $f$ to $g$ any morphism $u:V\rightarrow W$ in $\mathcal C$ such that the following diagram commutes.
\[\xymatrix{V\ar@{ >->}[rd]_-{f}\ar[rr]^-{u}&&W\ar@{ >->}[ld]^-{g}\\&E}\]
By Proposition \ref{strict}, $u$ is necessarily a strict monomorphism. Moreover, $u$ is the unique morphism satisfying $g=f\circ u$ because $g$ is a monomorphism. If $u$ is an isomorphism, we say $V$ and $W$ are isomorphic as strict sub-objects of $E$.
\end{defn}

Let $f:V\rightarrow E$ be a strict monomorphism. If there is no ambiguity on the morphism $f$, for convenience we also say that $V$ is a strict sub-object of $E$. We emphasize that an underlying strict monomorphism is always considered when we talk about a strict sub-object.

 \begin{egs}\label{eg.admissible}
 \begin{enumerate}
 \item Let $X$ be a non-empty set,
 \[\Gamma=\{F : \text{$F$ is a subset of $X$}\}.\]
 Then, $(\Gamma,\subseteq )$ is a bounded lattice.
 \item \label{Item. proto}
 Let $E$ be a non-zero object in a small proto-abelian category $\mathcal{C}$. The equivalent classes of isomorphic strict sub-objects of $E$ form a set $\Gamma$. Let $V,W$ be two elements of $\Gamma$, $E_1$ (resp. $E_2$) be a representative of $V$ (resp. $W$). We denote $V$ by $[E_1]$ and $W$ by $[E_2]$. We say $V\leqslant W$ if there is a morphism from $E_1\rightarrowtail E $ to $E_2\rightarrowtail E$.
  One can check that this definition is independent of the choice of representatives and it defines a partial order relation over the set $\Gamma$.
 Moreover, $(\Gamma, \leqslant)$ is a bounded lattice since $[E]$ is the largest element of $\Gamma$, $[\boldsymbol{0}]$ is the smallest element of $\Gamma$, $[E_1\cap E_2]=[E_1]\wedge [E_2]$ and
     $[E_1+E_2]=[E_1]\vee [E_2]$ by Proposition \ref{properties} \ref{Item:universal}, where $E_1,E_2$ are strict sub-objects of $E$.
 \end{enumerate}
 \end{egs}

The following definitions could be viewed as the parallel version of what we established in proto-abelian category with regard to bounded lattices.

\begin{defn}
Let $\Gamma$ be a bounded lattice, $E$ be the largest element of $\Gamma$. We say $(E_1,E_2)$ be an admissible pair of $\Gamma$ if $E_1,E_2$ are elements in $\Gamma$ and $E_1\leqslant E_2$. We define a partially ordered set denoted by $\mathcal S'_\Gamma$, the underlying set is defined as follows:
\[\{(E_1,E_2): \text{$(E_1,E_2)$ is an admissible pair of $\Gamma$} \}.\]
If $\mathcal W_1=(W_1',W_1)$ , $\mathcal W_2=(W_2',W_2)$ are two elements in $\mathcal S'_\Gamma$,
 we say $\mathcal W_1 \leqslant \mathcal W_2$ if both $W_1'\leqslant W_2'$ and $W_1\leqslant W_2$ hold.
We denote by $\mathcal S_\Gamma^{'0}$ the set
 \[\{(E_1,E_2): \text{$(E_1,E_2)$ is an admissible pair of $\Gamma$ and $E_1\neq E_2$} \}.\] Moreover, the partial order relation is inherited from $\mathcal S'_\Gamma$.
\end{defn}

Let $E$ be a nonzero object of a small proto-abelian category $\mathcal C$, $\Gamma$ be the bounded lattice defined in Examples \ref{eg.admissible} (\ref{Item. proto}). The partially ordered set $\mathcal S'_\Gamma$ viewed as a category is equivalent to the category $\mathcal S_E$.

\begin{defn}\label{defn:slope inequality set}
Let $\Lambda$ be a  partially ordered set, $\Gamma$ be a bounded lattice and $E$ be the largest element of $\Gamma$. We call \emph{slope function of $\Gamma$ valued in $\Lambda$} any map $\mu:\mathcal S_\Gamma^{'0}\rightarrow\Lambda$ which satisfies the following
slope inequality: \\
 Let $(W'_1,W_1)$  and   $(W'_2,W_2)$
 be elements in $\mathcal S_\Gamma^{'0}$ satisfying $W'_1=W'_2\wedge W_1$ and $W_2=W_1\vee W_2'$,
then $\mu(W'_1,W_1)\leqslant\mu(W'_2, W_2)$. \\
If furthermore, we delete the condition $W_2=W_1\vee W_2'$, we also have $\mu(W'_1,W_1)\leqslant\mu(W'_2, W_2)$. We say the slope function satisfies the strong slope inequality.\\
We denote by $\mu(V/W)$ the element $\mu(W,V)$ to make the notation consistent. It should be emphasised that $V/W$ is just a notation which does not have a quotient meaning.
\end{defn}

\begin{rmk}
Let $E$ be an object in a small proto-abelian category $\mathcal C$, $\Lambda$ be a partially ordered set. Assume that $\mu$ is a slope function defined on $\mathcal S^0_E$, then it induces a slope function over $\mathcal S^{'0}_\Gamma$ and vice versa.
\end{rmk}

\section{Semi-stability and Harder-Narasimhan filtration}

In this section, we fix a small proto-abelian category $\mathcal C$ and a partially ordered set  $(\Lambda,\leqslant)$. As usual, for any $(x,y)\in \Lambda^2$,
\begin{enumerate}[label=\rm(\arabic*)]
\item $x\geqslant y$ denotes the condition $y\leqslant x$.
\item  $x<y$ denotes the condition $x\leqslant y$ but $x\neq y$.
\item $x>y$ denotes the condition $y<x$.
\end{enumerate}

\begin{assum}\label {assumption}
We assume that $\Lambda$ admits  a greatest element $+\infty$, and that any subset $A$ of $\Lambda$ admits an infimum, which we denote by $\inf A$. Recall that $\inf A$ is by definition an element of $\Lambda$ which satisfies the following properties:
\begin{enumerate}[label=\rm(\alph*)]
\item $\inf A$ is a lower bound of $A$, namely for any $a\in A$ one has $\inf A\leqslant a$.
\item if $m$ is an element of $\Lambda$ such that $\forall\,a\in A$, $m\leqslant a$, then $m\leqslant\inf A$.
\end{enumerate}
\end{assum}

\subsection{Minimal slope}
In this subsection, we give a definition of the minimal slope and give a property of it (see Proposition \ref{Pro: minimal slope inequality}).
\begin{defn}{\label{defn:slopemin}}
Given a \emph{non-zero} object $E$ of $\mathcal C$ and a slope function $\mu:\mathcal S_E^0\rightarrow\Lambda$, we call \emph{minimal slope of $E$ and} we denote by $\mu_{\min}(E)$ the value
\[\inf_{\pi:E\twoheadrightarrow G}\mu^{\pi}(\xymatrix{0\ar[r]&G\ar[r]^-{\operatorname{id}_G}&G}),\]
 where $\pi:E\twoheadrightarrow G$ runs over the set of all non-zero strict quotient objects of $E$. If $E$ is a zero object, by convention the minimal slope of $E$ is defined to be $+\infty$.\end{defn}

\begin{rmk}\label{slopemin}
Let $\pi:E\rightarrow G$ be a strict epimorphism, and $\eta:F\rightarrow E$ be a kernel of $\pi$. By definition one has $\mu^\pi(G/\boldsymbol{0})=\mu(E/F)$ (see Notation \ref{Not: mu of subquotient}). Hence,
\[\mu_{\min}(E)=\inf_{F\rightarrowtail E}\mu(E/F)\]
where $F\rightarrowtail E$ runs over the set of proper strict sub-objects of $E$.
\end{rmk}

We could define $\mu_{\min}$ for slope functions of bounded lattices $\Gamma$.
\begin{defn}\label{defn:slope min set}
Given a bounded lattice $\Gamma$ and a slope function $\mu:\mathcal S_\Gamma^{'0}\rightarrow\Lambda$. Assume that $E_1,E_2\in \Gamma$ satisfying $E_2< E_1$, we call \emph{minimal slope of $E_1/E_2$} and we denote by $\mu_{\min}(E_1/E_2)$ the value
\[\inf_{F\in \Gamma, E_2\leqslant F< E_1}\mu(E_1/F).\]
By Remark \ref{slopemin}, Definition \ref{defn:slope min set} is compatible with Definition \ref{defn:slopemin}.
\end{defn}

\begin{prop}[Minimal slope inequality]\label{Pro: minimal slope inequality}
Let $\Gamma$ be a bounded lattice, $E_1,E_2,H$ be elements in $\Gamma$ satisfying $E_2\leqslant H< E_1$, then for any slope function $\mu:\mathcal S_E^{'0}\rightarrow\Lambda$ one has
\[\mu_{\min}(E_1/E_2)\leqslant\mu_{\min}(E_1/H).\]
\end{prop}

\begin{proof}
By Definition \ref{defn:slope min set},
\[\mu_{\min}(E_1/H)=\inf_{H\leqslant H'< E_1}\mu(E_1/H')\geqslant \]
\[\inf_{E_2\leqslant H'< E_1}\mu(E_1/H')=\mu_{\min}(E_1/E_2).\]
Therefore, $\mu_{\min}(E_1/H)\geqslant \mu_{\min}(E_1/E_2)$.
\end{proof}

\subsection{Harder-Narasimhan filtration}

\begin{prop}\label{main}
Let $\Gamma$ be a bounded lattice, $E$ be the largest element of $\Gamma$ and $\mu$ be a slope function of $\Gamma$ satisfying strong slope inequality.
If $H$, $V$ and $W$ are elements of $\Gamma$ satisfying $H< V$ and $H< W$, then the following inequality holds:
\[\mu_{\min}((V\vee W)/H)\geqslant \inf\{\mu_{\min}(V/H),\mu_{\min}(W/H)\}.\]
\end{prop}
\begin{proof}
By definition,
\[\mu_{\min}((V\vee W)/H)=\inf_{H\leqslant U< (V\vee W)}\mu( (V\vee W)/U).\]
We claim that, either $U\wedge V< V$ or $U\wedge W< W$ holds. In fact, otherwise both $V\leqslant U$ and $W\leqslant U$ hold.  Therefore, $V\vee W\leqslant U$. However, by assumption $U<V\vee W$ which leads to a contradiction.
Without loss of generality, assume that $U\wedge V< V$.
Since both $U\leqslant V\vee W$ and $V\leqslant V\vee W$ hold,
by the strong slope inequality of Definition \ref{defn:slope inequality set}, one has
 \[\mu( (V\vee W)/U)\geqslant \mu(V/(U\wedge V)).\]
Moreover, since $U\wedge V\geqslant H$,
 \[\mu( V/(U\wedge V))\geqslant \mu_{\min}(V/H)\geqslant\inf\{\mu_{\min}(V/H),\mu_{\min}(W/H)\}.\]
Hence we obtain that \[\mu((V\vee W)/U)\geqslant\inf\{\mu_{\min}(V/H),\mu_{\min}(W/H)\}.\]
Taking the infimum with respect to $U$, we obtain the required inequality.
\end{proof}

The following Noetherian and Artinian conditions could be viewed as a generalization of existence of a rank function in the classical Harder-Narasimhan theory.

\begin{defn}{\label{defn:NT and AT}}Let $E$ be an object of $\mathcal C$. We say that $E$ satisfies the \emph{Noetherian condition} if for any chain
\[\xymatrix{E_0\ar[r]^-{f_0}&E_1\ar[r]^-{f_1}&E_2\ar[r]&\cdots\cdots\ar[r]&E_n\ar[r]^-{f_n}&E_{n+1}\ar[r]&\cdots\cdots}\]
of morphisms of strict sub-objects of $E$, $f_n$ is an isomorphism for all sufficiently large $n$. We say that $E$ satisfies the \emph{Artinian condition} if for any chain
\[\xymatrix{\cdots\cdots\ar[r]&F_{n+1}\ar[r]^-{g_n}&F_n\ar[r]&\cdots\cdots\ar[r]&F_2\ar[r]^-{g_1}&F_1\ar[r]&F_0}\]
of morphisms of strict sub-objects of $E$, $g_n$ is an isomorphism for all sufficiently large $n$.
\end{defn}

\begin{rmk}\label{rem: noetherian artinian}
Let $\pi:E\rightarrow G$ be a strict epimorphism in $\mathcal C$. By Remark \ref{Rem: pull-back which is an iso}, we obtain that, if $E$ satisfies the Noetherian condition (resp. Artinian condition), so does $G$.
\end{rmk}

\begin{eg}
Let $M$ be a Noetherian (resp. Artinian) module over a commutative ring $R$. If we regard $M$ as an object of the category of $R$-modules, it satisfies the Noetherian condition (resp. Artinian condition). Note that, if $R$ is a field, then any finite-dimensional vector space over $R$ satisfies both Noetherian and Artinian conditions.\\
$\boldsymbol{NB}$: The Noetherian condition (resp. Artinian condition) in Definition \ref{defn:NT and AT} is a Noetherian condition (resp. Artinian condition) for strict monomorphisms. In proto-abelian category, a monomorphism may not be a strict monomorphism.
\end{eg}

\begin{defn}
Let $\Gamma$ be a bounded lattice. We say that $\Gamma$ satisfies the \emph{Noetherian condition} if for any chain
\[E_0\leqslant E_1\leqslant E_2\leqslant \cdots\cdots \leqslant E_n\leqslant E_{n+1}\leqslant\cdots\cdots\]
of elements in $\Gamma$, $E_{n}=E_{n+1}=\cdots$ for sufficiently large $n$. We say that $\Gamma$ satisfies the \emph{Artinian condition} if for any chain
\[\cdots\cdots\leqslant F_{n+1}\leqslant F_n\leqslant\cdots\cdots\leqslant F_2\leqslant F_1\leqslant F_0\]
 of elements in $\Gamma$, $F_n=F_{n+1}=\cdots$ for sufficiently large $n$.
\end{defn}

Now, we give a definition of semi-stability.

\begin{defn}\label{def: semi-stability}
Let $E$ be a non-zero object of $\mathcal C$ and $\mu:\mathcal S_E^0\rightarrow\Lambda$ be a slope function. We say that $E$ is \emph{semi-stable} with respect to the slope function $\mu$ if, for any strict monomorphism $\eta:F\rightarrowtail E$ such that $F$ is non-zero,
$\mu_{\min}(F)$ is not strictly larger than $\mu_{\min}(E)$, namely the condition
\[\neg(\mu_{\min}(E)<\mu_{\min}(F))\]
holds.
Note that, for each strict subquotient $H$ of $E$, there is a slope function of $H$ induced by $\mu$ as explained in Remark \ref{Rem: induced slope}. For simplicity we say that $H$ is \emph{semi-stable} with respect to $\mu$ if it is semi-stable with respect to the induced slope function.
\end{defn}

 Assume that $\Lambda$ is a totally ordered set, $E$ is semi-stable with respect to the slope function $\mu$ if and only if for any strict monomorphism $\eta:F\rightarrowtail E$ such that $F$ is non-zero, one has
\[\mu_{\min}(F)\leqslant \mu_{\min}(E).\]

\begin{defn}
Let $\Gamma$ be a bounded lattice, $E$ be the largest element of $\Gamma$ and $\mu:\mathcal S_\Gamma^{'0}\rightarrow\Lambda$ be a slope function. We say that an admissible pair $E_1/E_2$ is \emph{semi-stable} with respect to the slope function $\mu$ if, for any $F\in \Gamma$  satisfying $E_2< F\leqslant E_1$,
$\mu_{\min}(F/E_2)$ is not strictly larger than $\mu_{\min}(E_1/E_2)$, namely the condition
\[\neg(\mu_{\min}(E_1/E_2)<\mu_{\min}(F/E_2))\]
holds. The definition is consistent with Definition \ref{def: semi-stability} because of Proposition \ref{prop:compactilblity}.
\end{defn}

\begin{lemma}\label{main2}
Let $\Gamma$ be a bounded lattice, $E$ be the largest element of $\Gamma$ and 
$\mu:\mathcal S^{'0}_\Gamma\rightarrow \Lambda$ be a slope function satisfying strong slope inequality. 
Assume that $\Gamma$ satisfies Noetherian and Artinian conditions. Then for any $(N,M)\in \mathcal S_\Gamma^{'0}$, there exists an element $E_1\in \Gamma$ such that $N< E_1\leqslant M$ and $E_1$ satisfies the following conditions:
\begin{enumerate}[label=\rm(\arabic*)]
\item $E_1/N$ is semi-stable.
\item If $G\in \Gamma$ satisfies both $N< G\leqslant M$ and $\mu_{\min}(G/N)\geqslant \mu_{\min}(E_1/N)$, then $G\leqslant E_1$.
\end{enumerate}
In particular, $\mu_{\min}(E_1/N)$ is maximal in the set \begin{equation*}\{\mu_{\min}(W/N): N< W\leqslant M\; \}.\end{equation*}

\end{lemma}
\begin{proof}
The statement is trivial when $M/N$ is semi-stable with respect to $\mu$ (it suffices to take $E_1=M$). In the following, we suppose that $M/N$ is not semi-stable, which means that there exists at least one element $F_1\in \Gamma$ such that $N< F_1\leqslant M$ and  $\mu_{\min}(F_1/N)>\mu_{\min}(M/N)$. We claim that there exists an element $W_1\in \Gamma$ which satisfies the following properties:
\begin{enumerate}[label=\rm(\arabic*)]
\item $N< W_1$.
\item $\mu_{\min}(M/N)<\mu_{\min}(W_1/N)$.
\item If $W_1'\in \Gamma$ satisfies both $N< W_1'\leqslant M$ and $\mu_{\min}(W_1/N)\leqslant\mu_{\min}(W_1'/N)$, then $W_1$ is not strictly less than $W_1'$.
\end{enumerate} Otherwise we can construct a sequence
\[F_1< F_2<\cdots\cdots< F_n< F_{n+1}<\cdots,\]
   such that each $F_i\leqslant M$ and \[\mu_{\min}(F_1/N)\leqslant\mu_{\min}(F_2/N)\leqslant\cdots\leqslant\mu_{\min}(F_n/N)\leqslant \cdots.\]
This contradicts the Noetherian condition on $\Gamma$.

We now replace $M/N$ by $W_1/N $ and iterate the above construction. By the Artinian condition on $\Gamma$ we obtain a finite sequence of elements in $\Gamma$,
\[W_k\leqslant W_{k-1}\leqslant \cdots\cdots\leqslant W_1\leqslant W_0=M,\]
such that $W_k/N$ is semi-stable,
\[\mu_{\min}(W_k/N)>\ldots>\mu_{\min}(W_1/N)>\mu_{\min}(M/N),\]
and for any $i\in\{1,\ldots,k\}$, there is no $W'_{i}$ satisfying both $W_i< W'_i\leqslant W_{i-1}$ and $\mu_{\min}(W_i/N)\leqslant \mu_{\min}(W_i'/N)$.

  Assume that $G\in \Gamma$ satisfies both $N< G\leqslant M$ and $\mu_{\min}(G/N)\geqslant \mu_{\min}(W_k/N)$. We prove by induction that $G\leqslant W_i$, where $0\leqslant i\leqslant k$.  Clearly the statement is true when $i=0$. Assume that $G\leqslant W_{i-1}$ holds. By Proposition  \ref{main} one has  \[\mu_{\min}((W_{i}\vee G)/N)\geqslant\inf\{\mu_{\min}(W_{i}/N),\mu_{\min}(G/N)\}=\mu_{\min}(W_{i}/N).\]
 Since $W_i\leqslant W_{i}\vee G \leqslant W_{i-1}$, one has $W_{i}= W_{i}\vee G$. In other words, $G\leqslant W_{i}$. Take $i=k$, one has $G\leqslant W_k$. Hence $W_k$ satisfies the required conditions.
 Moreover, $\mu_{\min}(G/N)=\mu_{\min}(W_k/N)$ since $W_k/N$ is semi-stable. Therefore, $\mu_{\min}(W_k/N)$ is a maximal element in the set $\{\mu_{\min}(W/N): N< W\leqslant M \}$ . \\
\end{proof}

\begin{defn}Let $E$ be a non-zero object of $\mathcal C$. We call a \emph{filtration} of $E$ a finite sequence of morphisms in $\mathcal{C}$, \[\xymatrix{\boldsymbol{0}=E_0\ar@{ >->}[r]& E_1\ar@{ >->}[r]&\cdots\cdots\ar@{ >->}[r]& E_n=E,}\] such that $E_i$ is a proper strict sub-object of $E_{i+1}$ for any $i\in\{0,\ldots,n-1\}$.
Let \[\boldsymbol{0}=E_0\rightarrowtail E_1\rightarrowtail \cdots \rightarrowtail E_n=E\quad
\text{and}\quad
\boldsymbol{0}=F_0\rightarrowtail F_1\rightarrowtail \cdots\rightarrowtail F_m=E\] be filtrations of $E$. We say they are isomorphic if
the equality $m=n$ holds, and there exists isomorphisms $f_i:E_i\rightarrow F_i$, $i\in\{0,\ldots,n\}$ satisfying $f_n=\operatorname{id}_E$ and making the following diagrams commute,
 \[\xymatrix{E_{i-1}\ar[d]_-{f_{i-1}}\ar@{ >->}[r]&E_{i}\ar[d]^-{f_{i}}\\
F_{i-1}\ar@{ >->}[r]&F_{i}}\]
where $i\in\{1,\ldots,n\}$.
\end{defn}

\begin{defn}Let $\Gamma$ be a bounded lattice and $E$ be the largest element of $\Gamma$. We call a \emph{filtration} of $E$ a finite sequence of elements in $\Gamma$, \[\boldsymbol{0}=E_0< E_1< \cdots\cdots< E_n=E,\]
where $\boldsymbol{0}$ is the smallest element of $\Gamma$.
\end{defn}

\begin{thm}\label{existence}(Existence of Harder-Narasimhan filtration)
Let $E$ be a non-zero object of a small proto-abelian category $\mathcal{C}$, which satisfies Noetherian and Artinian conditions, and let $\mu:S^0_E\rightarrow \Lambda$ be a slope function satisfying the strong slope inequality.
Then there exists a filtration of $E$,
\[\xymatrix{0=E_0\ar@{ >->}[r]& E_1\ar@{ >->}[r]&\cdots\cdots\ar@{ >->}[r]& E_n=E,}\]
which satisfies the following properties:
\begin{enumerate}[label=\rm(\arabic*)]
\item\label{Item: semi-stable subquotient} For any $i\in\{1,\ldots,n\}$, $E_{i}/E_{i-1}$ is semi-stable.
\item\label{Item: mu min is maximal element} Assume that $E_{i-1}\rightarrowtail E$ factorizes through strict sub-object $G_i\rightarrowtail E$. If $\mu_{\min}(G_i/E_{i-1})\geqslant\mu_{\min}(E_i/E_{i-1})$, then $G_i\rightarrowtail E$ factorizes through $E_i\rightarrowtail E$.
In particular, each $\mu_{\min}(E_i/E_{i-1})$ is a maximal element of the set of $\mu_{\min}(F_i/E_{i-1})$, 
where $F_i/E_{i-1}$ is a non-zero strict sub-objects of $E/E_{i-1}$, and for any decomposition $E_{i-1}\rightarrowtail F_i\rightarrowtail E$ of the strict sub-object $E_{i-1}\rightarrowtail E$ such that $E_{i-1}\rightarrowtail F_i$ is not an isomorphism, one has $\mu_{\min}(E_i/E_{i-1})\not\leqslant\mu_{\min}(F_i/E_{i-1})$.

\item One has
\begin{equation}\label{Equ: not increasing}
\forall\,i\in\{1,\ldots,n-1\},\quad\mu_{\min}(E_i/E_{i-1})\not\leqslant\mu_{\min}(E_{i+1}/E_{i}).\end{equation}
In particular, if $\Lambda$ is a totally ordered set, then the following inequalities hold:
\[\mu_{\min}(E_1/E_0)>\mu_{\min}(E_2/E_1)>\ldots>\mu_{\min}(E_n/E_{n-1}).\]
\end{enumerate}

\end{thm}

We prove Theorem \ref{existence} in a more general setting.
\begin{thm}\label{existence in set}
Let $\Gamma$ be a bounded lattice and $E$ be the largest element of $\Gamma$. Assume that $\Gamma$ satisfies Noetherian and Artinian conditions, and let $\mu:\mathcal S^{'0}_\Gamma\rightarrow \Lambda$ be a slope function satisfying the strong slope inequality. Then
there exists a filtration
\[\boldsymbol{0}=E_0< E_1< \cdots\cdots< E_n=E,\]
which satisfies the following properties:
\begin{enumerate}
\item\label{Item: semi-stable subquotient in set} For any $i\in\{1,\ldots,n\}$, $E_{i}/E_{i-1}$ is semi-stable.
\item\label{Item: mu min is maximal element in set} For any element $G_i\in \Gamma$ satisfying  $E_{i-1} < G_i$, if $\mu_{\min}(G_i/E_{i-1})\geqslant\mu_{\min}(E_i/E_{i-1})$, then $G_i\leqslant E_i$.
In particular, each $\mu_{\min}(E_i/E_{i-1})$ is a maximal element of the set of  $\mu_{\min}(F_i/E_{i-1})$ where $E_{i-1}< F_i$, which means for any $F_i$ satisfying $E_{i-1}< F_i$, one has $\mu_{\min}(E_i/E_{i-1})\not\leqslant\mu_{\min}(F_i/E_{i-1})$.

\item One has
\begin{equation}\label{Equ: not increasing in set}
\forall\,i\in\{1,\ldots,n-1\},\quad\mu_{\min}(E_i/E_{i-1})\not\leqslant
\mu_{\min}(E_{i+1}/E_{i}).\end{equation}
In particular, if $\Lambda$ is a totally ordered set, then the following inequalities hold:
\[\mu_{\min}(E_1/E_0)>\mu_{\min}(E_2/E_1)>\ldots>\mu_{\min}(E_n/E_{n-1}).\]
\end{enumerate}
\end{thm}
\begin{proof}

Note that  $\Gamma$ satisfies both Noetherian and Artinian conditions, thus Lemma \ref{main2} allows us to construct recursively a filtration
\[\boldsymbol{0}=E_0< E_1< \cdots\cdots< E_n=E,\]
where $\boldsymbol{0}$ is the smallest element of $\Gamma$,
 and the filtration satisfies the properties (\ref{Item: semi-stable subquotient in set}) and  (\ref{Item: mu min is maximal element in set}). Note that the recursive procedure terminates in finite steps, thanks to the Noetherian condition on $\Gamma$.

To show the inequality \eqref{Equ: not increasing in set}, by replacing $E/E_0$ by $E/E_{i-1}$, we may assume without loss of generality that $i=1$. 
We only need to show $\mu_{\min}(E_1/E_{0})\not\leqslant
\mu_{\min}(E_{2}/E_{1})$. Suppose by contradiction that
$\mu_{\min}(E_1/E_0)\leqslant\mu_{\min}(E_2/E_{1})$ holds.
  Let $U$ be an element of $\Gamma$ satisfying $E_0\leqslant U< E_2$. If $E_1\leqslant U$, by Proposition \ref{Pro: minimal slope inequality}, one has \[\mu_{\min}(E_2/E_1)\leqslant\mu_{\min}(E_2/U)\leqslant \mu(E_2/U).\]
  Thus, $\mu_{\min}(E_1/E_0)\leqslant \mu(E_2/U)$. If $E_1\not\leqslant U$, then  $U\wedge E_1< E_1$.  By the strong slope inequality, we obtain that
\[\mu(E_1/(U\wedge E_1))\leqslant\mu(E_2/U),\]
which still leads to $\mu_{\min}(E_1/E_0)\leqslant\mu(E_2/U)$. Taking the infimum with respect to $U$, we obtain that the inequality $\mu_{\min}(E_1/E_0)\leqslant\mu_{\min}(E_2/E_0)$ holds, which leads to a contradiction.
\end{proof}

\begin{defn}\label{ HN filtration}
Let $E$ be a non-zero object of the proto-abelian category $\mathcal{C}$,  $\mu:S^0_E\rightarrow \Lambda$ be a slope function, where $\Lambda$ is a partially ordered set satisfying assumption \ref{assumption}. Let
 \[\boldsymbol{0}=E_0\rightarrowtail E_1\rightarrowtail \cdots \rightarrowtail E_n=E\]
 be a filtration of $E$ such that $E_1/E_0,\ldots,E_{n}/E_{n-1}$ are semi-stable and
\[\mu_{\min}(E_1/E_0)>\mu_{\min}(E_2/E_1)>\ldots>\mu_{\min}(E_n/E_{n-1}).\]
We call the filtration
\[\boldsymbol{0}=E_0\rightarrowtail E_1\rightarrowtail \cdots \rightarrowtail E_n=E\]
 a Harder-Narasimhan filtration of $E$.

\end{defn}
\begin{eg}

In this example, we explain that the condition "strong slope inequality" in Theorem \ref{existence}(2) is necessary. Recall Examples \ref{examples} (\ref{item: Counterexample slope function}), we defined a proto-abelian category $\mathcal C$.
\[\xymatrix{
                &\boldsymbol{ 0} \ar[dl]\ar[d]\ar[dr] &            \\
 L  \ar@{>->}[r]_-f &M\ar@{->>}[r]_-g &N             }\]
 We define a function $\mu:\mathcal S^0_M\rightarrow \mathbb{R}$ by $\mu(L)=2, \mu(M)=1$ and $\mu(N)=3$. Let $W, W'$ be strict sub-objects of $M$, thus $(W'+W)/W'$ is isomorphic to $W/(W'\cap W)$ by enumerating all possible situations.  Therefore, $\mu$ satisfies the slope inequality. However, notice that $0+L=L$ is a strict sub-object of $M$ and $\mu(M)<\mu(L)$, $\mu$ does not satisfy the strong slope inequality. Observe that both $\mu_{\min}(M)=1$ and $\mu_{\min}(L)=2$ hold, so $M$ is not semi-stable. The only filtration could be a Harder-Narasimhan filtration is
 \[\boldsymbol{0}\rightarrowtail L\rightarrowtail M.\]
  However, note that $\hat{\mu}_1=\mu_{\min}(L)=2$ and $\hat{\mu}_2=\mu_{\min}(M/L)=3$, one has $\hat{\mu}_2>\hat{\mu}_1$ which means the filtration does not satisfy the decreasing condition in Theorem \ref{existence} (2).
\end{eg}

\begin{defn}\label{ HN filtration in set}
Let $\Gamma$ be a bounded lattice, $E$ be the largest element of $\Gamma$, and $\mu:\mathcal S^{'0}_\Gamma\rightarrow \Lambda$ be a slope function, where $\Lambda$ is a partially ordered set satisfying assumption \ref{assumption}. Let
 \[\boldsymbol{0}=E_0< E_1< \cdots< E_n=E\]
 be a filtration of $E$ such that $E_1/E_0,\ldots,E_{n}/E_{n-1}$ are semi-stable and
\[\mu_{\min}(E_1/E_0)>\mu_{\min}(E_2/E_1)>\ldots>\mu_{\min}(E_n/E_{n-1}).\]
We call the filtration
\[\boldsymbol{0}=E_0< E_1< \cdots< E_n=E\]
 a Harder-Narasimhan filtration of $E$.

\end{defn}

\begin{thm}\label{Uniqueness}(Uniqueness of Harder-Narasimhan filtration)
Let $E$ be a non-zero object in a small proto-abelian category $\mathcal{C}$ satisfying Noetherian and Artinian conditions, $(\Lambda,\leqslant)$ be a \emph{totally ordered} set which satisfies Assumption \ref{assumption}, and $\mu:\mathcal S^0_E\rightarrow \Lambda$ be a slope function. If there exists a Harder-Narasimhan filtration of $E$, it is unique up to isomorphism.
\end{thm}

We prove Theorem \ref{Uniqueness} in a more general setting.
\begin{thm}\label{Uniqueness in set}
Let $\Gamma$ be a bounded lattice, $E$ be the largest element of $\Gamma$, $(\Lambda,\leqslant)$ be a \emph{totally ordered} set which satisfies Assumption \ref{assumption}, and $\mu:\mathcal S^{'0}_\Gamma\rightarrow \Lambda$ be a slope function of $\Gamma$. Assume that $\Gamma$  satisfies Noetherian and Artinian conditions.  If there exists a Harder-Narasimhan filtration of $E$, it is unique.
\end{thm}
\begin{proof}
Let
 \[\boldsymbol{0}=E_0< E_1< \cdots< E_n=E\quad
\text{and}\quad
\boldsymbol{0}=F_0< F_1< \cdots< F_m=E\]
be two Harder-Narasimhan filtrations of $E$.
Suppose that $E_j\neq F_j$ for some $j$ and let $k$ be the smallest of such $j$. Since $\Lambda$ is a totally ordered set, one has either $\mu_{\min}(E_k/E_{k-1})\leqslant\mu_{\min}(F_k/F_{k-1})$ or $\mu_{\min}(F_k/F_{k-1})\leqslant\mu_{\min}(E_k/E_{k-1})$ holds. 
Without loss of generality, we may suppose that $\mu_{\min}(F_k/F_{k-1})\leqslant\mu_{\min}(E_k/E_{k-1})$. Let $j\in\{k,\ldots,m\}$ be the least index such that $E_k\leqslant F_j$, and $U$ be an element in $\Gamma$ satisfying $F_{j-1}\leqslant U< F_{j-1}\vee E_k$. Thus $E_k\wedge U< E_k$. Moreover, since $E_{k-1}=F_{k-1}\leqslant F_{j-1}\leqslant U$, one has $E_{k-1}\leqslant E_k\wedge U$. By definition, one has
 \[\mu_{\min}(E_k/E_{k-1})\leqslant \mu(E_k/(E_k\wedge U)).\]
 Observe that $F_{j-1}\vee E_k=U\vee E_k$, using the slope inequality, one finds that
 \[\mu(E_k/(E_k\wedge U))\leqslant \mu((F_{j-1}\vee E_k)/U).\]
  Taking the infimum with respect to $U$, one has
 \[\mu_{\min}(E_k/E_{k-1})\leqslant \mu_{\min}((F_{j-1}\vee E_k)/F_{j-1}).\]
Since $F_j/F_{j-1}$ is semi-stable,
\[\mu_{\min}(E_k/E_{k-1})\leqslant \mu_{\min}((F_{j-1}\vee E_k)/F_{j-1})\leqslant \mu_{\min}(F_j/F_{j-1}).\]
If $j>k$, then by definition of Harder-Narasimhan filtration,  
\[\mu_{\min}(F_j/F_{j-1})<\mu_{\min}(F_k/F_{k-1}).\] 
Thus  $\mu_{\min}(E_k/E_{k-1})< \mu_{\min}(F_k/F_{k-1})$, which contradicts the assumption that $\mu_{\min}(F_k/F_{k-1})\leqslant \mu_{\min}(E_k/E_{k-1})$. Hence $j=k$ which implies that $E_k\leqslant F_k$ and $\mu_{\min}(E_k/E_{k-1})=\mu_{\min}(F_k/F_{k-1})$. Switching the two filtrations, by applying the above argument we obtain that $F_k= E_k$ which is opposite to the assumption.
Therefore, $n=m$ and $E_i=F_i$ for any $1\leqslant i\leqslant n$.

\end{proof}

\begin{rmk}
In this remark, we give an example that shows that the total order assumption is necessary in Theorem \ref{Uniqueness}. Let $\mathcal{C}$ be the category of finite abelian groups, which is an abelian category, $E$ be a nonzero finite abelian group, $\Lambda$ be the partially ordered set defined as follows:

$\Lambda=\{A: A\subseteq \mathbb{N}\}$ and
$A\leq B \Longleftrightarrow A\supseteq B$.\\
 We define a map $\mu: \mathcal S^0_E\rightarrow \Lambda$ as follows:
 \[\mu(\xymatrix{W'\ar@{ >->}[r]&W\ar@{ >->}[r]&E})=\begin{cases}
 \{p\}, & \text{if $ \ord(W/W')=p^n$},\\
 \emptyset, & \text{otherwise},
 \end{cases}
  \]
  where $p$ is a prime, $\ord(W/W')$ is the order of the finite group $W/W'$.
 Note that for any finite abelian group $F$,
 \[\mu_{\min}(F)=\{ p: \text{ $p|\ord(F)$}\}.\]
  Therefore, the nonzero finite abelian group $F$ is semi-stable if and only if $\ord(F)=p^n$.
   Assume that $\ord(E)$ is not a power of a prime, one has $\mu_{\min}(E)=\{p_1,\cdots, p_m\}$ where $m\geq 2$.  By the structure theorem of finite abelian group, $E$ is isomorphic to $\bigoplus_{k=1}^m E_{p_k}$, where $E_{p_k}$ is the Sylow $p_k$-subgroup of $E$. Hence the following filtrations,
   \[0\rightarrowtail E_{p_1}\rightarrowtail E_{p_1}\bigoplus E_{p_2} \rightarrowtail\cdots \rightarrowtail E\quad
\text{and}\quad
0\rightarrowtail E_{p_2}\rightarrowtail E_{p_1}\bigoplus E_{p_2}\rightarrowtail \cdots\rightarrowtail E,\]
   are not isomorphic because $E_{p_1}$ is not isomorphic to $E_{p_2}$.
\end{rmk}

\begin{eg}
Let $E$ be a  Hermitian space, $\mathscr{A}$ be a  Hermitian transform on $E$, $h$ be the minimal polynomial of $\mathscr{A}$ and $\mathcal{C}_h$ be the proto-abelian category defined in Examples \ref{examples} (\ref{Item: spectral decomposition}).
 We define the slope function $\mu:\mathcal S^0_{(E, \mathscr{A})}\rightarrow \mathbb{R}$ as follows, for any non-zero strict subquotient $(H, \mathscr{A}_H)$ of $(E, \mathscr{A})$,
  \[\mu(H, \mathscr{A}_H)=\max\{\lambda: \text{$\lambda$ is an eigenvalue of $\mathscr{A}_H$ on H}\}.\]
  Hence $\mu$ satisfies the strong slope inequality by definition. Observe that
  \begin{equation*}
  \begin{aligned}
  \mu_{\min}(H, \mathscr{A}_H)=&\inf\{\mu(G,\mathscr{A}_G): \text{$(G,\mathscr{A}_G)$ is a non-zero strict quotient of $(H,\mathscr{A}_H)$}\}\\
  =&\min\{\lambda: \text{$\lambda$ is an eigenvalue of $\mathscr{A}_H$ on $H$}\}. \\
  \end{aligned}
  \end{equation*}
  Hence the object $(H,\mathscr{A}_H)$ is semi-stable if and only if $\mathscr{A}_H$ has a unique eigenvalue on $H$.
   Let $\{\lambda_1,\cdots,\lambda_n\}$ be the set of eigenvalues of $\mathscr{A}$, where $\lambda_1>\cdots>\lambda_n$,
  and let
  \[\xymatrix{\boldsymbol{0}=(E_0, \mathscr{A}_{E_0})\ar@{ >->}[r]& (E_1, \mathscr{A}_{E_1})\ar@{ >->}[r]&\cdots\cdots\ar@{ >->}[r]& (E_n, \mathscr{A}_{E_n})=(E, \mathscr{A}),}\]
  be the Harder-Narasimhan filtration of $E$, where $\hat{\mu}_i=\mu_{\min}(E_i/E_{i-1}, \mathscr{A}_{E_i/E_{i-1}})$.
   One has, for any $1\leq i\leq n$, $\lambda_i=\hat{\mu}_i$ and  \[E_i=\bigoplus_{j\leqslant i} V_j,\] where $V_j$ is the eigenspace associated with the eigenvalue $\lambda_j$.
  Therefore, the Harder-Narasimhan filtration of $(E, \mathscr{A})$ is the spectral decomposition of the Hermitian transform $\mathscr{A}$.
\end{eg}
\subsection{Compatibility}
In this subsection, we want to make a comparison between the Harder-Narasimhan filtration constructed in Theorem \ref{existence} and that in the classical articles such as \cite{MR2571693} where the slope function is defined as a ratio of a degree function and a rank function in the frame of proto-abelian category.
\begin{defn}
A \emph{rank function} on a small category $\mathcal{C}$ is a function\\
\[\rk:Obj(\mathcal{C})\rightarrow \mathbb{N}\]
where $\mathbb{N}$ is the set of nonnegative integers, satisfying the following conditions:
\begin{enumerate}
 \item $\rk(X)=\rk(X')$ if $X$ and $X'$ are isomorphic.
 \item $\rk(L)+\rk(N)=\rk(M)$ for any short exact sequence $0\rightarrow L\rightarrow M\rightarrow N\rightarrow 0$.
 \item $\rk(X)=0$ if and only if $X$ is a zero object in $\mathcal{C}$.
 \end{enumerate}
\end{defn}

 Let $\mathcal{C}$ be a small proto-abelian category, and assume that there is a rank function on $\mathcal{C}$. 
 If $E$ is an object of $\mathcal{C}$,  then $E$ satisfies both Noetherian and Artianian conditions.\\

\begin{defn}
Let $\mathcal{C}$ be a small category. A \emph{degree function} on $\mathcal{C}$ is a function\\
\[\deg:Obj(\mathcal{C})\rightarrow \mathbb{R},\]
 satisfying the following conditions:
\begin{enumerate}
\item $\deg(X)=\deg(X')$ if $X$ and $X'$ are isomorphic.
\item $\deg(M)=\deg(L)+\deg(N)$ for any  short exact sequence $0\rightarrow L\rightarrow M\rightarrow N\rightarrow 0$.
\end{enumerate}
\end{defn}

Let $\mathcal{C}$ be a small proto-abelian category, $\rk$ be a rank function on $\mathcal{C}$.
The slope function is defined by
\[\mu(M)=\frac{\deg(M)}{\rk(M)},\]
for any nonzero object $M\in Obj(\mathcal{C})$. Moreover, we assume that
\begin{equation} \label{Equ: Assumption}
\mu(N/(N\cap H))\leqslant \mu((N+H)/H),
\end{equation}
 where $N,H$ are strict sub-objects of some object $E$.

 \begin{rmk}
In the work of Andr\'e \cite{MR2571693}, he assumes that if there exists a morphism $f:M\rightarrow N$ in a proto-abelian category $\mathcal{C}$ such that $f$ is not only an epimorphism, but also a monomorphism, then
$\mu(M)\leqslant \mu(N)$. This assumption is stronger than (\ref{Equ: Assumption}).\\
\end{rmk}

  Let $F$ be an object of $\mathcal C$. We define \[\mu_{\max}(F)=\sup_{F'\rightarrowtail F} \mu(F'),\] where $F'$ is a nonzero strict sub-object of $F$. Assume that $N,H,W$ are strict sub-objects of some object $E$, thus one has $\mu_{\max}(N/(N\cap H))\leqslant \mu_{\max}((N+H)/H)$. Moreover, $\mu_{\max}((N+H)/H)\leqslant \mu_{\max}(W/H)$ if $N+H\rightarrowtail E$ factorizes through $W\rightarrowtail E$, hence $\mu_{\max}(N/(N\cap H))\leqslant \mu_{\max}(W/H)$. We define the slope function of $E$ valued in $\mathbb{R}$ (see Definition \ref{Def: slope function}) as follows,
  \[\mu^E(\mathcal Y)=\mu_{\max}(F/F'),\]
  where $\mathcal Y=(\xymatrix{F'\ar@{ >->}[r]^-{u}&F\ar@{ >->}[r]^-{v}&E})$ is an element
in $\mathcal S^0_E$.  Therefore, $\mu^E(N/N\cap H)\leqslant \mu^E(W/H)$, which is the "strong slope inequality".

We call a strict sub-object $E'_1$ of $E$ a universal destabilizing sub-object of $E$ if the following conditions hold:
\begin{enumerate}
 \item $\mu(E_1')=\mu_{\max}(E)$.
 \item Let $H\rightarrow E$ be a strict monomorphism. If $\mu(H)=\mu(E_1')$, then $H\rightarrow E$ factorizes through $E'_1 \rightarrowtail E$.
 \end{enumerate}
The universal destabilizing sub-object always exists and is unique up to a unique isomorphism (see \cite[p. 22]{MR2571693}). Moreover, $\mu(E_1')> \mu_{\max}(E/E_1')$ holds (see \cite[p.25]{MR2571693}).

Let $V$ be a strict sub-object of $E$. We say $V$ is $\mu$-\emph{semi-stable}, if for any strict sub-object $F$ of $V$, $\mu(F)\leqslant \mu(V)$. In particular, the universal destabilizing sub-object $E_1'$ is $\mu$-semi-stable. Moreover, if $V$ is $\mu$-semi-stable, then $\mu(V/F)\geqslant \mu(V)$, where $F$ is a strict sub-object of $V$.
\begin{prop}\label{slope minima and slope}
Assume that $V$ is a strict sub-object of $E$, and it is $\mu$-semi-stable, then
\[\mu(V)=\mu^E_{\min}(V).\]
\end{prop}
\begin{proof}
 By definition, one has
 \[\mu^E_{\min}(V)\leqslant\mu^E(V)=\mu_{\max}(V)=\mu(V).\]
 Furthermore, for any nonzero strict quotient $V/F$ of $V$,
 \[\mu^E(V/F)=\mu_{\max}(V/F)\geqslant \mu(V/F)\geqslant \mu(V).\]
 The second inequality holds because $V$ is $\mu$-semi-stable. Thus $\mu(V)\leqslant \mu^E_{\min}(V)$ by taking the infimum with respect to $F$. Therefore,  $\mu^E_{\min}(V)=\mu(V)$.
\end{proof}
In Theorem \ref{existence}, we find a Harder-Narasimhan filtration of $E$,
\[\boldsymbol{0}=E_0\rightarrowtail E_1\rightarrowtail \cdots \rightarrowtail E_n=E.\]  The strict sub-object $E_1$ of $E$ satisfies \[\mu^E_{\min}(E_1)=\sup_{W\rightarrow E} \mu^E_{\min}(W),\]
where $W$ is a nonzero strict sub-object of $E$. One could see that $\mu^E_{\min}(E_1)\leqslant \mu(E_1')$ because
 \[\mu^E_{\min}(E_1)\leqslant \mu^E(E_1)= \mu_{\max}(E_1)\leqslant\mu_{\max}(E)=\mu(E_1').\]
 On the other hand, \[\mu^E_{\min}(E_1)\geqslant \mu^E_{\min}(E_1')=\mu(E'_1),\] by Theorem \ref{existence} \ref{Item: mu min is maximal element} and Proposition \ref{slope minima and slope}. As a result,  $\mu^E_{\min}(E_1)=\mu(E_1')$.
  Since $\mu^E_{\min}(E_1)=\mu(E'_1)=\mu^E_{\min}(E_1')$, $E_1'\rightarrowtail E$ factorizes through $E_1\rightarrowtail E$ by $h_1:E'_1\rightarrowtail E_1$ by Theorem \ref{existence} \ref{Item: mu min is maximal element}. If $h_1$ is not an isomorphism,  then
  \[\mu_{\max}(E_1/E_1')\leqslant \mu_{\max}(E/E_1')<\mu(E_1'),\]
  which leads to a contradiction since \[\mu(E_1')=\mu^E_{\min}(E_1)\leqslant \mu^E(E_1/E_1')=\mu_{\max}(E_1/E_1').\]
  Therefore, $h_1$ is an isomorphism.

Let $E'_{i+1}/E'_i$ be the universal destabilizing sub-object of $E/E'_i$. One has a filtration \[\boldsymbol{0}=E'_0\rightarrowtail E'_1\rightarrowtail\cdots \rightarrowtail E'_m=E.\]
 Just replace $E_1$ by $E_{i+1}/E_i$, $E_1'$ by $E_{i+1}'/E_i'$ and use the same method, there is an isomorphism $h_{i+1}:E'_{i+1}\rightarrow E_{i+1}$ such that the following diagram commutes.
 \[\xymatrix{E'_{i+1}\ar@{ >->}[rd]\ar[rr]^-{h_{i+1}}&&E_{i+1}\ar@{ >->}[ld]\\&E}\]Therefore,
 the two filtrations
 $\bold{0}=E_0'\rightarrowtail E_1'\rightarrowtail \cdots\rightarrowtail E'_m=E $
 and  $\bold{0}=E_0\rightarrowtail E_1\rightarrowtail \cdots\rightarrowtail E_n=E $
   are isomorphic and $\mu(E'_{i+1}/E'_i)=\mu^E_{\min}(E_{i+1}/E_i)$.

   \section*{acknowledgement}
   This article is accomplished under the tutelage of Professor Huayi Chen. Firstly, I should thank him deeply for his repeatedly reading this article and ceaselessly giving me advice. The motivation of this article is also from a discuss with him. Moreover, Professor Cornut pointed out partial order on the set of strict sub-objects may help to understand the Harder-Narasimhan filtration. I am grateful for his suggestions. I would also express my gratitude for his/her careful reading and his/her suggestions for the writing of the article.

 \bibliographystyle{Plain}
 \bibliography{refer}
 \end{document}